\documentclass[12pt]{article}

\usepackage{graphicx}
\usepackage{latexsym,amsmath,amsfonts,amscd, amsthm, dsfont}
\usepackage{bm,color}
\usepackage{epsfig,verbatim,epstopdf,graphics}
\usepackage{subfigure}
\usepackage{changebar}
\usepackage{multirow}
\usepackage{amssymb}
\usepackage{geometry}
\usepackage{verbatim}
\usepackage{fancyhdr}
\usepackage{graphicx}
\usepackage{cases}
\usepackage{listings}
\usepackage{tabu}
\usepackage{extarrows}
\graphicspath{{./}{./figure/}}
\allowdisplaybreaks

\newtheoremstyle{plainNoItalics}{}{}{\normalfont}{}{\bfseries}{.}{ }{}

\theoremstyle{plain}
\newtheorem{thm}{Theorem}[section]

\theoremstyle{plainNoItalics}

\newtheorem{rem}[thm]{Remark}

\newtheorem{exa}[thm]{Example}

\renewcommand{\theequation}{\thesection.\arabic{equation}}
\renewcommand{\thefigure}{\thesection.\arabic{figure}}
\renewcommand{\thetable}{\thesection.\arabic{table}}

\newcommand{\mD}{\mathcal{D}}
\newcommand{\mH}{\mathcal{H}}
\newcommand{\mI}{\mathcal{I}}
\newcommand{\mL}{\mathcal{L}}
\newcommand{\mO}{\mathcal{O}}

\makeatletter

\newcommand{\Rmnum}[1]{\expandafter\@slowromancap\romannumeral #1@}
\makeatother

\begin{document}



\begin{center}
{\bf
A Kernel Based  Unconditionally Stable Scheme for Nonlinear Parabolic Partial Differential Equations 
}
\end{center}

\vspace{.2in}
\centerline{
Kaipeng Wang \footnote{
	School of Mathematical Sciences, University of Science and Technology of China, Hefei, Anhui, 230026, People's Republic of China. 
	{\tt kpwang@mail.ustc.edu.cn} 
},
Andrew Christlieb \footnote{
	Department of Computational Mathematics, Science and Engineering, Michigan State University, East Lansing, MI, 48824, United States.
	{\tt christli@msu.edu}.
	Research is supported in part by AFOSR grants FA9550-12-1-0343, FA9550-12-1-0455, and FA9550-15-1-0282, and NSF grant DMS-1418804.
},
Yan Jiang \footnote{
	School of Mathematical Sciences, University of Science and Technology of China, Hefei, Anhui, 230026, People's Republic of China.
	{\tt  jiangy@ustc.edu.cn}
	Research supported by NSFC grant 11901555.
}, and 
Mengping Zhang \footnote{
	School of Mathematical Sciences, University of Science and Technology of China, Hefei, Anhui, 230026, People's Republic of China.
	{\tt mpzhang@ustc.edu.cn}. 
	Research supported by NSFC grant 11871448.}
}

\bigskip

{\bf Abstract.}
In this paper, a class of high order numerical schemes is proposed to solve the nonlinear parabolic equations with variable coefficients. 
This method is based on our previous work \cite{christlieb2017kernel} for convection-diffusion equations, which relies on a special kernel-based formulation of the solutions and successive convolution. However, disadvantages appear when we extend the previous method to our equations, such as inefficient choice of parameters and unprovable stability for high-dimensional problems.  To overcome these difficulties, a new kernel-based formulation is designed to approach the spatial derivatives. It maintains the good properties of the original one, including the high order accuracy and unconditionally stable for one-dimensional problems, hence allowing much larger time step evolution compared with other explicit schemes.
In additional, without extra computational cost, the proposed scheme can enlarge the available interval of the special parameter in the formulation, leading to less errors and higher efficiency. 
Moreover, theoretical investigations indicate that it is unconditionally stable for multi-dimensional problems as well.
We present numerical tests for one- and two-dimensional scalar and system, demonstrating the designed high order accuracy and unconditionally stable property of the scheme.

\bigskip

{\bf Key Words:} Nonlinear parabolic equation, kernel based scheme, unconditionally stable, high order accuracy 

\section{Introduction}

In this work, we want to solve the nonlinear parabolic equations
\begin{align}
\label{eq:equation}
	\partial_t u(\mathbf{x},t) = \nabla \cdot \left( \textbf{A}(u, \textbf{x}, t) \nabla u\right) + \textbf{B}(u, \textbf{x}, t)^T \nabla u + C(u, \textbf{x}, t),
\end{align}
on the domain $\mathbf{x}\in \Omega\subset \mathbb{R}^n$ with initial and boundary conditions.
Here, $\textbf{A}(u, \textbf{x}, t)=\left( A_{ij}(u, \textbf{x}, t) \right) \in \mathbb{R}^{n\times n}$ and $\textbf{B}(u, \textbf{x}, t) = \left( B_{i}(u, \textbf{x}, t) \right) \in \mathbb{R}^{n}$.
In particular, the equation \eqref{eq:equation} is parabolic if there exists a constant $\theta>0$ such that 
	$$\allowdisplaybreaks \sum_{i,j=1}^{n} A_{ij}\xi_{i}\xi_{j} \geq \theta \sum_{i=1}^{n} \xi_{i}^{2}, \qquad 
	\forall \, (\xi_{1},\cdots,\xi_{n})\in \mathbb{R}^{n}. $$

For such a time-dependent partial differential equation (PDE) \eqref{eq:equation}, one common method is splitting the equation into a system with an auxiliary variable $\mathbf{w}\in \mathbb{R}^{n}$ at first, 
\begin{align}
\label{eq:system}
	\left\{\begin{array}{l}
	\partial_t u(\mathbf{x},t) = \nabla \cdot \left( \mathbf{A}(u, \textbf{x}, t) \mathbf{w} \right) + \mathbf{B}(u, \textbf{x}, t)^T \nabla u + C(u, \textbf{x}, t),\\
	\mathbf{w} = \nabla u. \\
\end{array}\right.
\end{align}
And then solve the two equations at the same time level. 
There is a large amount of numerical methods for this problem.
Most of these schemes discretize the spatial variables at first with finite volume / difference methods, finite element methods, or spectral methods, generating a large coupled system of ordinary differential equations (ODEs). And then apply an initial value ODE solver in time. This approach is commonly referred to as the Method of Lines (MOL) and interested readers are referred to \cite{schiesser2012numerical} for further discussions.
Classical methods for this time evolution include multi-step, multi-stage, or multi-derivative methods, as well as a combination of these approaches. For instance, the Runge-Kutta method and the Taylor series methods.
Note that efficiency is a main concern of these schemes. 
For example, the explicit methods solving \eqref{eq:system} does restrict the time step $\Delta t\varpropto \Delta x^2$ due to the stability requirement, where $\Delta x$ is the spatial mesh size. Using Implicit-Explicit (IMEX) or fully implicit time discretization techniques \cite{ruuth1993implicit, kennedy2003additive} can allow larger time step, but usually we need to solve a system of (nonlinear) equations for each step. The algorithm would be expensive when the system size becomes bigger. 
	Besides the classical ones, other high order time discretization techniques were also developed, e.g., the spectral deferred correction (SDC) method \cite{dutt2000spectral, minion2003semi, huang2007arbitrary}, the exponential time differencing method \cite{cox2002exponential, kassam2005fourth}, the integration factor methods \cite{beylkin1998new, cox2002exponential, ju2015fast, maday1990operator, nie2006efficient}, and the hybrid methods of SDC and high order Runge-Kutta schemes \cite{christlieb2010integral}.

Another framework named the Method of Lines Transpose (MOL$^T$) has been exploited in the literature for solving the linear time-dependent PDEs. 
In such a framework, the temporal variable is first discretized, resulting in a set of linear boundary value problems (BVPs) at  discrete time levels. Furthermore, each BVP can be inverted analytically in an integral formulation based on a kernel function and then the numerical solution is updated accordingly. As a notable advantage, the MOL$^T$ approach is able to use an implicit method but avoid solving linear systems at each time step, see \cite{causley2014method}. Moreover, a fast convolution algorithm is developed to ensure the computational complexity of the scheme is $\mathcal{O}(N)$ \cite{causley2013method, greengard1987fast, barnes1986hierarchical}, where $N$ is the number of discrete mesh points. Over the past several years, the MOL$^T$ methods have been developed for solving the heat equation \cite{causley2017method, causley2016method, kropinski2011fast, jia2008krylov},  Maxwell's equations \cite{cheng2017asymptotic}, the advection equation and Vlasov equation \cite{christlieb2016weno}, among others. This methodology can be generalized to solving some nonlinear problems, such as the Cahn-Hilliard equation \cite{causley2017method}.  However, it rarely applied to general nonlinear problems, mainly because efficient fast algorithms of inverting nonlinear BVPs are lacking and hence the advantage of the MOL$^T$ is compromised.

More recently, following the MOL$^T$ philosophy, authors found that the first order and second order derivatives can be represented as infinite series of the kernel based integral \cite{christlieb2017kernel}. Therefore, in numerical simulations, we can truncate the series and use the corresponding partial sum to approximate the spatial derivatives. 
This method was presented to solve the nonlinear degenerate parabolic equations in \cite{christlieb2017kernel},
\begin{align}
\label{eq:original}
	u_t+f(u)_x=g(u)_{xx},
\end{align}
which is a special case of \eqref{eq:equation} with $A(u,x,t)=g'(u)$. The major distinction between
the kernel based scheme and the MOL$^T$ works is that this scheme is still in the MOL framework with an the classic explicit strong-stability-preserving Runge-Kutta (SSP RK) scheme in time discretization \cite{gottlieb2001strong, shu2002survey, gottlieb2005high}, which is stable, efficient and accurate. Even though the scheme is explicit, it was proved to be unconditionally stable up to third order accuracy, with the help of the careful choice of a parameter $\beta$ in it. 
After that, the scheme has been extended to the Hamilton-Jacobi equations \cite{christlieb2019kernel}, and applied on the ideal magnetohydrodynamics equations \cite{cakir2019kernel}.  
	We have tried to employ the scheme to solve \eqref{eq:system} directly. Unfortunately, the numerical scheme is less efficient, because the available interval for $\beta$ is pretty small, which would result larger errors. Even worse, for the two-dimensional problems, the unconditional stability is absent. Details would be shown later. 





In this paper, we will propose a numerical scheme to discretize \eqref{eq:equation} or \eqref{eq:system}, with a novel kernel-based representation of the spatial derivatives. Again, the scheme is in the MOL framework, coupled with the classic SSP RK method in time discretization. We want the scheme can maintain the good properties of unconditional stability and high order accuracy. In additional, comparing with the the original method \cite{christlieb2017kernel}, the novel scheme can enlarge the available interval for $\beta$, enhancing greater efficiency. Moreover, the unconditionally stable property for high dimensional problems can be proved theoretically. 
For ease of use in the following parts, we list the formulation of the SSP RK scheme here, up to third order accuracy. To advance the solution of the ODE $u_t=\mathcal{H}[u]$ at time level $t^n$, denoted by $u^n$, to next time level $t^{n+1}=t^n+\Delta t$, the first order scheme is the forward Euler scheme,
\begin{align}
\label{eq:rk1}
u^{n+1}=u^{n}+\Delta t \mathcal{H}[u^{n}].
\end{align} 
The second order scheme is 
\begin{equation}
\label{eq:rk2}
\begin{aligned}
& u^{(1)}=u^{n}+\Delta t \mathcal{H}[u^{n}], \\
& u^{n+1}=\frac{1}{2}u^{n}+\frac{1}{2} \left( u^{(1)}+\Delta t \mathcal{H}[u^{(1)}] \right).
\end{aligned}
\end{equation}
And the third order scheme is 
\begin{equation}
\label{eq:rk3}
\begin{aligned}
& u^{(1)}=u^{n}+\Delta t \mathcal{H}[u^{n}], \\
& u^{(2)}=\frac{1}{4}u^{n}+\frac{3}{4} \left( u^{(1)}+\Delta t \mathcal{H}[u^{(1)}] \right), \\
& u^{n+1}=\frac{2}{3}u^{n}+\frac{1}{3} \left( u^{(2)}+\Delta t \mathcal{H}[u^{(2)}] \right).
\end{aligned}
\end{equation}

The rest of this paper is organized as follows. In Section 2, we will review how the original kernel based formula works on \eqref{eq:original}. Then we will fix the method and discuss the properties of the new one, including accuracy and stability, in Section 3. Section 4 will introduce the two-dimensional approach. After that, we will present some numerical tests in Section 5 to verify the performance of our scheme, and finally, draw conclusions in Section 6.
\section{Representation of Differential Operators}
In this section, we will review the representations of the first spatial derivative $\partial_{x}$ given in \cite{christlieb2017kernel}. 
Such representations serve as the key building block of the proposed schemes. Below, we will introduce operator $\mathcal{L}$ and the corresponding operator  $\mathcal{D}$ at first. Then, the differential operator $\partial_x$ can be represented by an infinite series of $\mathcal{D}$. 
We will also investigate the approximation accuracy when the infinite series is truncated by a partial sum.

\subsection{The first derivative $\partial_{x}$}

In order to represent the first order derivative $\partial_{x}$, we start with two operators defined on a closed interval $x\in[a,b]$,  
\begin{align}
	\mathcal{L}_{L} := \mathcal{I}+\frac{1}{\alpha}\partial_{x} \quad \text{and} \quad 
	\mathcal{L}_{R} := \mathcal{I}-\frac{1}{\alpha}\partial_{x}, 
\end{align}
where $\mathcal{I}$ is the identity operator and $\alpha>0$ is a constant. Then, we can invert the operators analytically
\begin{align}
	\mathcal{L}_{L}^{-1} =I^{L}[v,\alpha](x) + A_{L}e^{-\alpha (x-a)},\qquad
	\mathcal{L}_{R}^{-1}=I^{R}[v,\alpha](x) + B_{R}e^{-\alpha (b-x)}
\end{align}
where, 
\begin{align}
\label{eq:I_LR}
	I^{L}[v,\alpha](x)=\alpha \int_a^x e^{-\alpha (x-y)}v(y)dy,\qquad
	I^{R}[v,\alpha](x)=\alpha \int_x^b e^{-\alpha (y-x)}v(y)dy
\end{align}
are the left/right biased integral, respectively.  
And $A_{L}$ and $B_{R}$ are constants determined by the boundary conditions. For instance, if $\mathcal{L}_{L}^{-1}$ and $\mathcal{L}_{R}^{-1}$ are periodic, that is $\mathcal{L}_{L}^{-1}(a)=\mathcal{L}_{L}^{-1}(b)$ and $\mathcal{L}_{R}^{-1}(a)=\mathcal{L}_{R}^{-1}(b)$, then $A_L=\frac{I^L[v,\alpha](b)}{1-\mu}$ and $B_R=\frac{I^R[v,\alpha](a)}{1-\mu}$ with $\mu=e^{-\alpha(b-a)}$.
Furthermore, the first order derivative $\partial_{x}$ can be represented as
\begin{align*}
	& \partial_x
	= \alpha ( \mathcal{L}_L - \mathcal{I})
	= \alpha \mathcal{L}_L ( \mathcal{I} - \mathcal{L}_L^{-1} ) 
	= \alpha (\mathcal{I}-\mathcal{D}_L)^{-1} \mathcal{D}_L
	= \alpha \sum_{p=1}^{\infty} \mathcal{D}^{p}_{L},\\
	& \partial_x
	= -\alpha ( \mathcal{L}_R - \mathcal{I})
	= -\alpha \mathcal{L}_R ( \mathcal{I} - \mathcal{L}_R^{-1} ) 
	= -\alpha (\mathcal{I}-\mathcal{D}_R)^{-1} \mathcal{D}_R
	=-\alpha\sum\limits_{p=1}^{\infty}\mathcal{D}_R^p,
\end{align*}
with
\begin{align}
\label{eq:DLR}
	\mathcal{D}_L=\mathcal{I}-\mathcal{L}_L^{-1},
	\quad\text{and}\quad 
	\mathcal{D}_R=\mathcal{I}-\mathcal{L}_R^{-1}.
\end{align} 
 
	



\noindent
In numerical simulations, we have to truncate the series and only compute the corresponding partial sums:
\begin{align}
\label{eq:partial_sum}
	\partial_{x} v \approx \alpha \sum_{p=1}^{k} \mathcal{D}_L^p[v,\alpha],
	\quad\text{or}\quad
	\partial_{x} v \approx -\alpha \sum_{p=1}^{k}\mathcal{D}_R^p[v,\alpha].
\end{align} 
In particular, in the case of periodic boundary condition considered in this work, it is natural to require the boundary treatment 
$\mathcal{D}^p_L(a)=\mathcal{D}^p_L(b)$ and $\mathcal{D}^p_R(a)=\mathcal{D}^p_R(b)$ for $p\geq1$.
Furthermore, \cite{christlieb2017kernel} showed the following theorem, which provided the expressions of the operators $\mathcal{D}_{*}$ and the error estimates of the corresponding $k$-th partial sums. The proof relies on integration by parts, and details can be found in \cite{christlieb2017kernel}.

\begin{thm}
\label{thm1}
Suppose $v(x)\in\mathit{C}^{k+1}([a,b])$ is a periodic function. If we employ the operator $\mathcal{D}_{*}$ with $\mathcal{D}_{*}(a)=\mathcal{D}_{*}(b)$, where $*$ can be L and R. 
\begin{enumerate}
	\item The operator $\mathcal{D}_{*}$ have the following expressions,
	\begin{subequations}
	\begin{align}
	\mathcal{D}_L[v,\alpha](x)=& -\sum\limits_{p=1}^k \left(-\frac{1}{\alpha}\right)^p\partial_x^p v(x) + \left(-\frac{1}{\alpha}\right)^{k+1} \mathcal{L}^{-1}_L[\partial_x^{k+1}v,\alpha](x),\\
	\mathcal{D}_R[v,\alpha](x)=& -\sum\limits_{p=1}^k \left(\frac{1}{\alpha}\right)^p\partial_x^p v(x) - \left(\frac{1}{\alpha}\right)^{k+1} \mathcal{L}^{-1}_R[\partial_x^{k+1}v,\alpha](x).
	\end{align}
	\end{subequations}


	\item Error between the partial sum and the first derivative can be bounded, 
	\begin{subequations}
	\begin{align}
		 \| \partial_{x}v - \alpha \sum_{p=1}^{k}\mathcal{D}_{L}^{p}[v,\alpha] \|_{\infty} \leq C
		\left(\frac{1}{\alpha}\right)^{k}  \|\partial^{k+1}_{x}v\|_{\infty},\\
		 \| \partial_{x}v + \alpha \sum_{p=1}^{k}\mathcal{D}_{R}^{p}[v,\alpha] \|_{\infty} \leq C
		\left(\frac{1}{\alpha}\right)^{k} \|\partial^{k+1}_{x}v\|_{\infty},
	\end{align}
	\end{subequations}
	where, the constant $C$ only depends on $k$.
\end{enumerate}
\end{thm}

Next, we look at the convergence performance of the partial sums \eqref{eq:partial_sum}. Here, we test with a special choice that $u=\sin x$, and plot the $L_\infty$ norm of errors in Figure \ref{figadd1}. It is obvious that the errors are always decreasing monotonely with $\alpha$, indicating the unform convergence of the partial sum \eqref{eq:partial_sum}.

\begin{figure}[htb]
	\centering
	\subfigure[$k=1$]{\includegraphics[width=.32\textwidth]{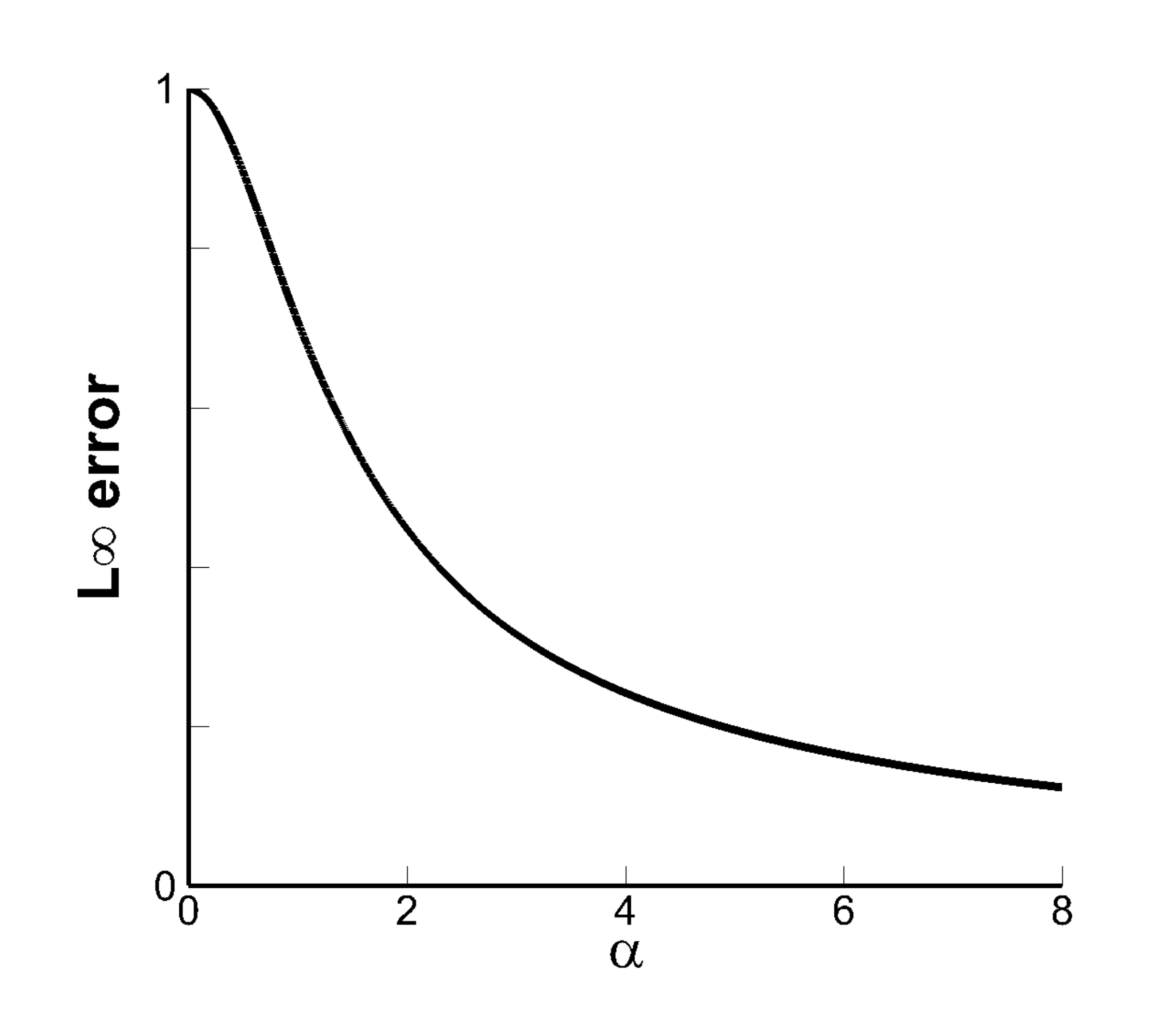}}
	\subfigure[$k=2$]{\includegraphics[width=.32\textwidth]{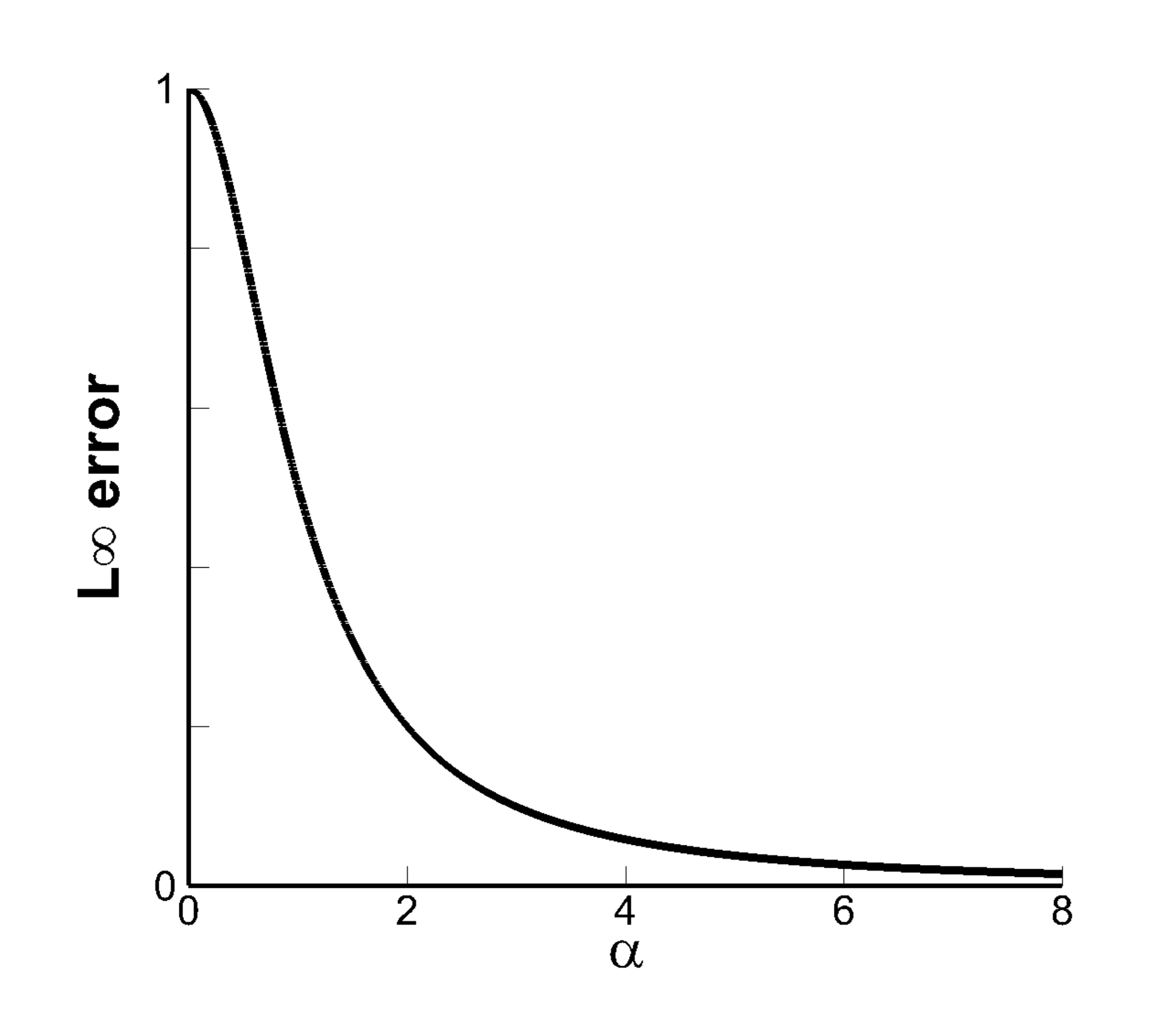}}
	\subfigure[$k=3$]{\includegraphics[width=.32\textwidth]{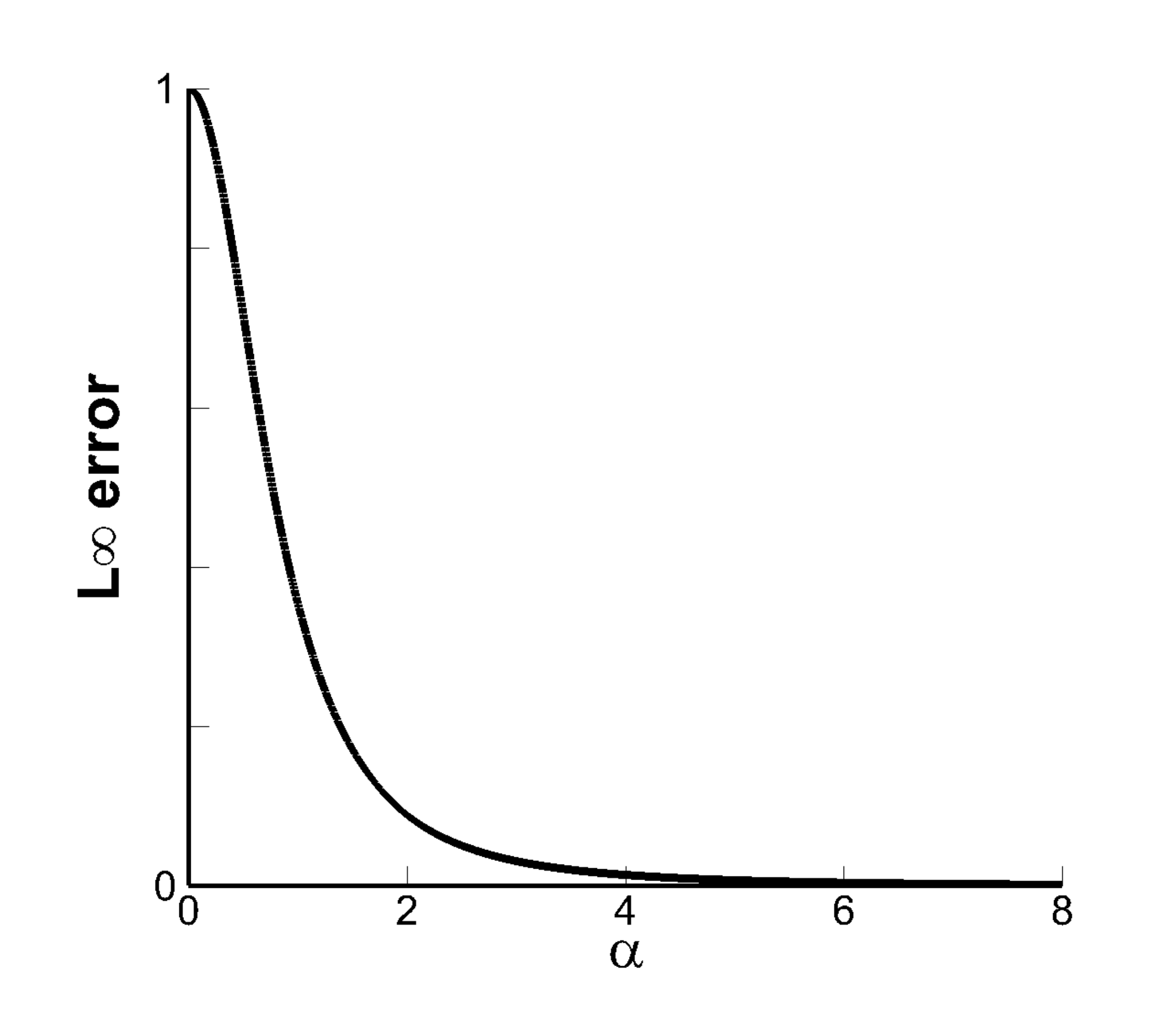}}\\
	\subfigure[$k=1$]{\includegraphics[width=.32\textwidth]{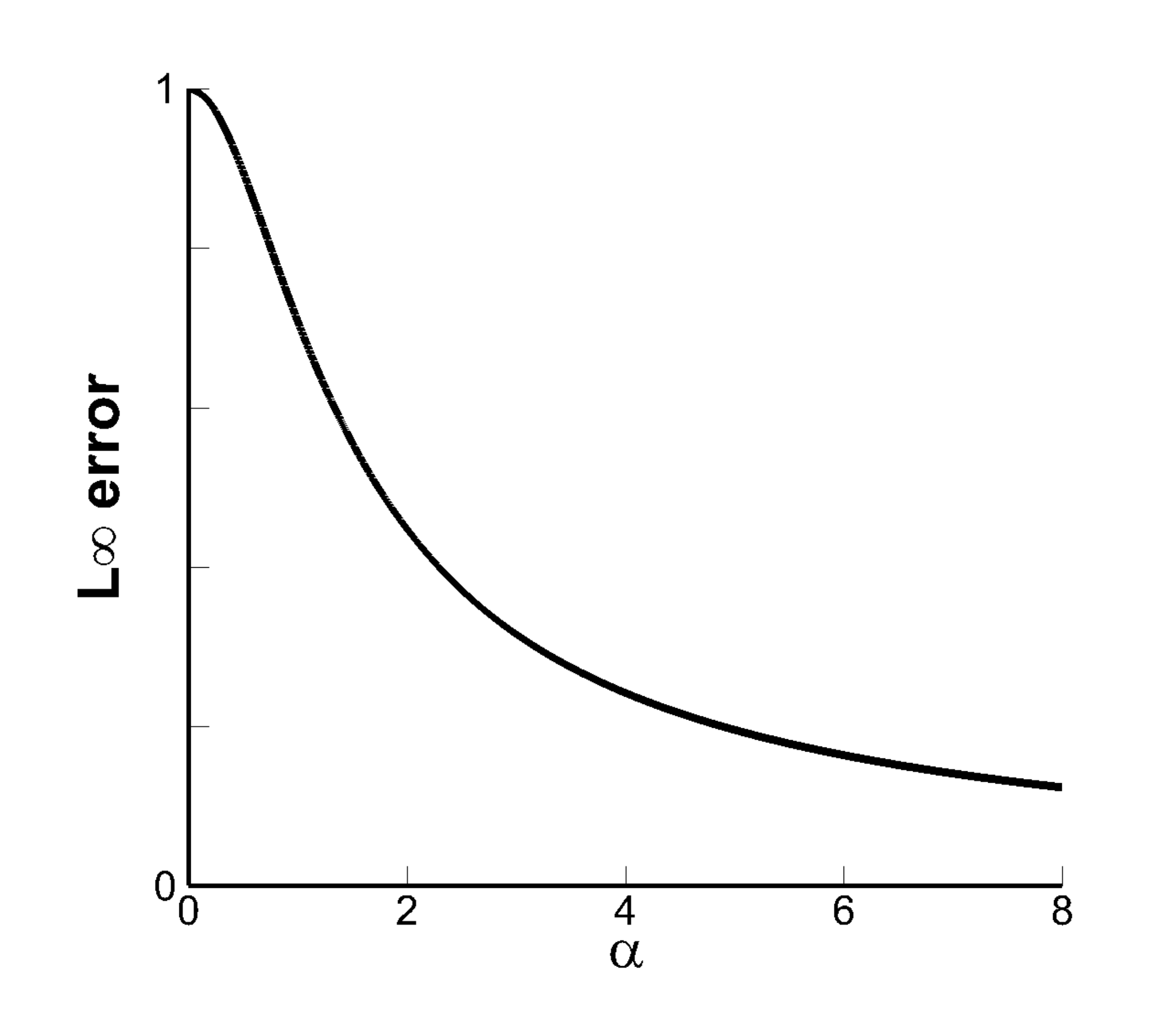}}
	\subfigure[$k=2$]{\includegraphics[width=.32\textwidth]{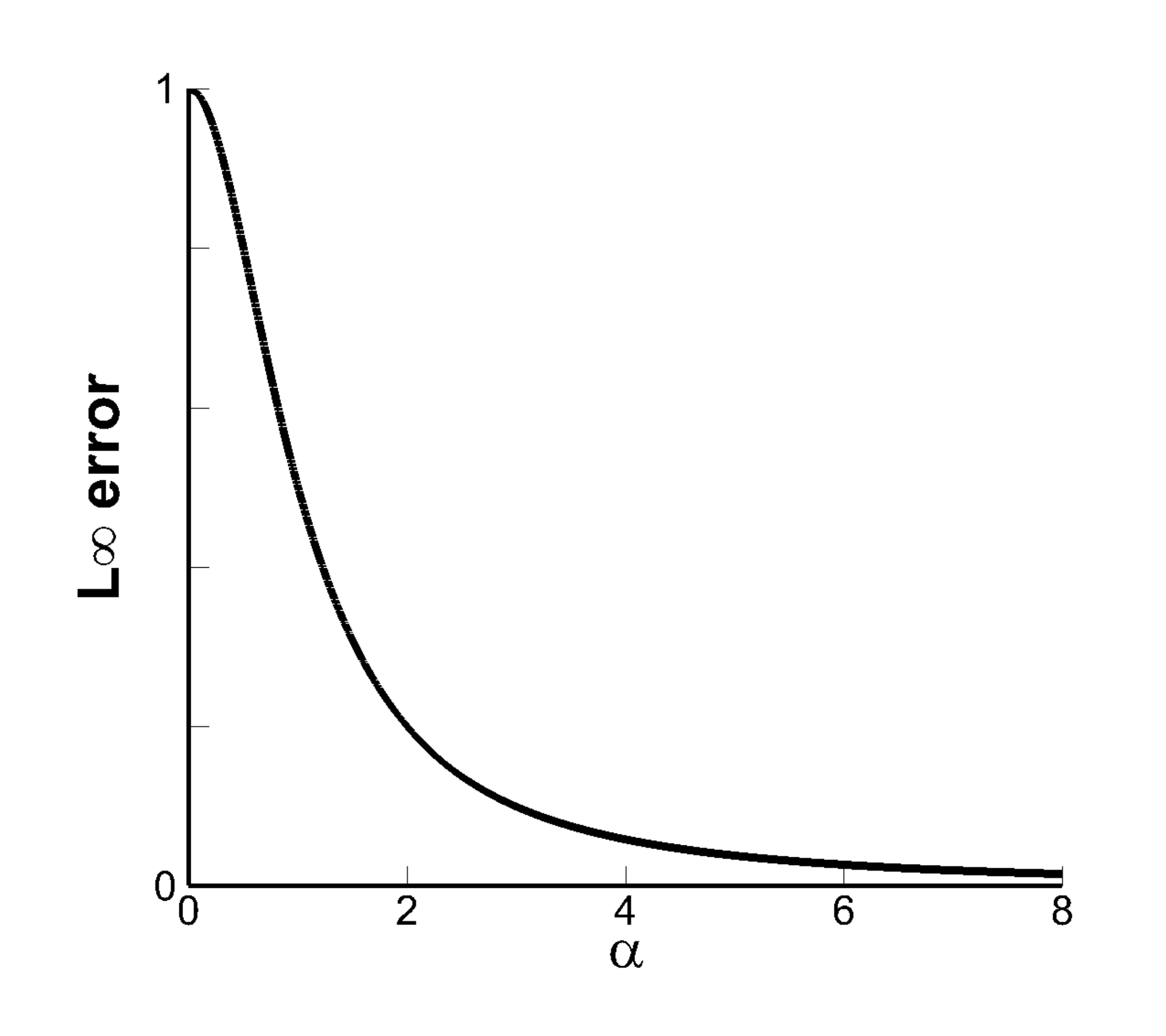}}
	\subfigure[$k=3$]{\includegraphics[width=.32\textwidth]{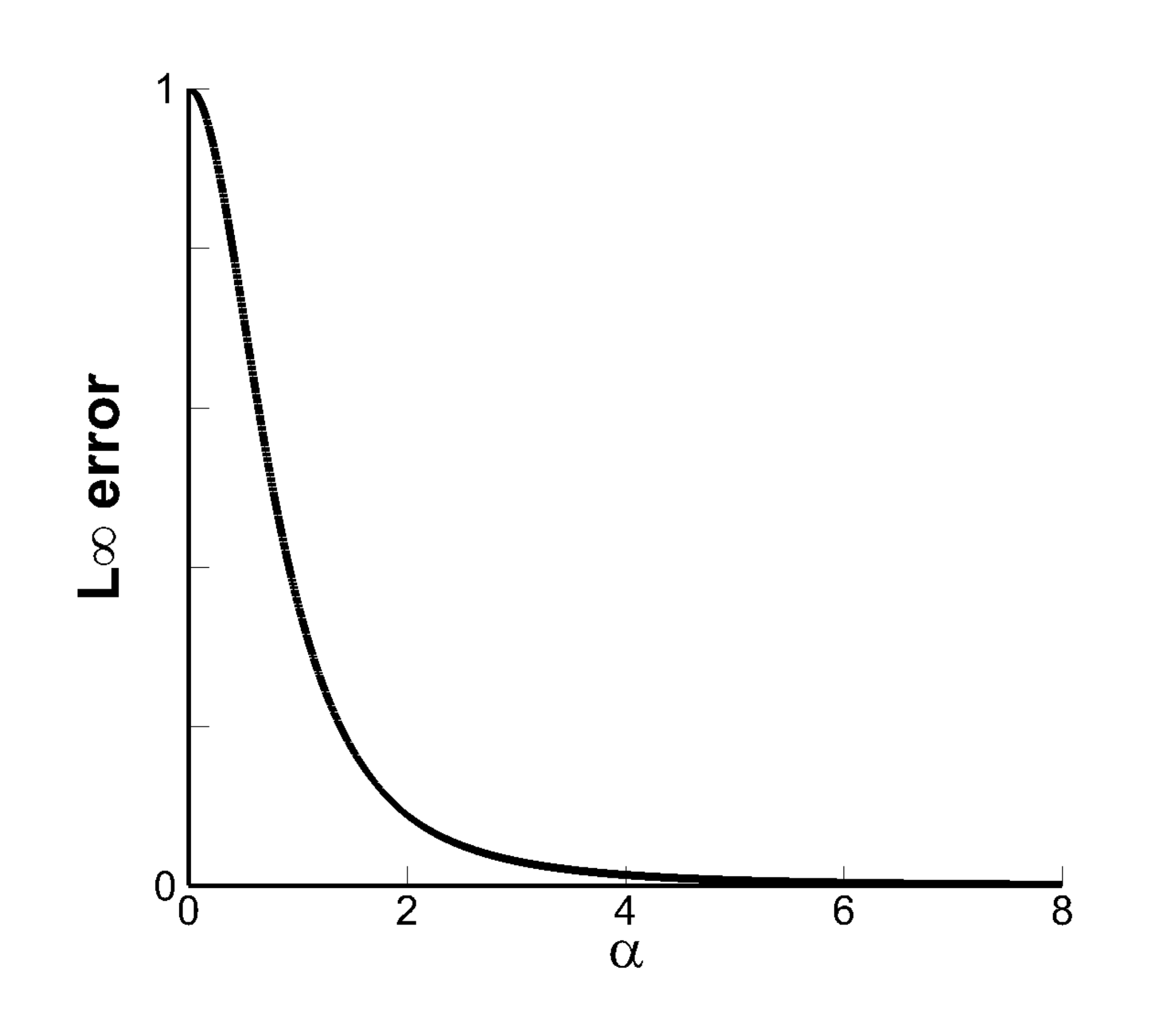}}
	\caption{Errors between the first derivative and the partial sum \eqref{eq:partial_sum} with $u=\sin x$. Top:  $\| \partial_{x}u - \alpha \sum_{p=1}^{k}\mathcal{D}_{L}^{p}[u,\alpha] \|_{\infty}$; bottom: $ \| \partial_{x}u + \alpha \sum_{p=1}^{k}\mathcal{D}_{R}^{p}[u,\alpha] \|_{\infty}$. }
	\label{figadd1}
\end{figure}

Hence, we can use the partial sum \eqref{eq:partial_sum} to approximate the transport term $B(u,x,t) u_x$,
\begin{align}
\label{eq:scheme_transport}
	B(u,x,t) \, u_x \approx 
	\frac{1}{2}B(u,x,t) \left( u^{-}_x + u_x^{+} \right) + \frac{1}{2}r \left( u^{+}_x - u_x^{-} \right), 
\end{align}
with $r=\max |B(u,x,t)|$ in the relevant range, and $u_{x}^{\pm}$ are the approximations to the derivative $u_x$ obtained by left-biased and right-biased methods, respectively, 
\begin{align}
\label{eq:ux_pn}
	u^{-}_x = \alpha_{L} \sum_{p=1}^{k} \mathcal{D}_{L}^{p} [u, \alpha_{L}] (x), \quad \text{and}\quad 
	u^{+}_x = -\alpha_{R} \sum_{p=1}^{k} \mathcal{D}_{R}^{p} [u,\alpha_{R}] (x).
\end{align}

In particular, we take $\alpha_{L} = \alpha_{R} =\beta/(q\Delta t)$, where $\Delta t$ denotes the time step and $\beta$ is a prescribed constant independent of $\Delta t$. Hence, the scheme \eqref{eq:scheme_transport} has an error $\mathcal{O}(\Delta t^k)$. Moreover, \cite{christlieb2017kernel} studied the stability of the semi-discrete scheme for scalar linear function $u_t=B u_x$, coupling \eqref{eq:scheme_transport} with exact integral and the classic explicit $k$-th order SSP RK methods in time. 
The scheme can be proved to be A-stable and hence allows large time step if $\beta$ in \eqref{eq:ux_pn} is appropriately chosen.

\begin{thm}\label{thm2}
	For linear function $u_t=B\,u_x$ with periodic boundary conditions, we use the classic explicit $k$-th order SSP RK methods for time evolution.
	\begin{enumerate}
		\item For $k=1, 2$, if we use \eqref{eq:scheme_transport} and \eqref{eq:ux_pn} to approximate the spacial derivative $u_{x}^{\pm}$, then the scheme is $k$-th order in time. Moreover,  there exists constant $\beta_{1,k,max} > 0$, such that the scheme is A-stable provided $0<\beta\leq\beta_{1,k,max}$.
		
		\item For $k=3$, we use \eqref{eq:scheme_transport} and a modified approximation to $u_x^{\pm}$,
		\begin{subequations}
		\label{eq:ux_pn_mod}
		\begin{align}
			& u^{-}_x = \alpha \sum_{p=1}^{3} \mathcal{D}_{L}^{p} [u, \alpha] (x) - \frac{1}{2} (\mathcal{D}_{L} +\mathcal{D}_{R}) \star \mathcal{D}_{L}^{2} [u, \alpha](x), \\
			& u^{+}_x = -\alpha \sum_{p=1}^{3} \mathcal{D}_{R}^{p} [u,\alpha] (x) + \frac{1}{2} (\mathcal{D}_{L} +\mathcal{D}_{R}) \star \mathcal{D}_{R}^{2} [u, \alpha](x).
		\end{align}
		\end{subequations}
		Then the scheme is third order in time. Moreover,  there exists constant $\beta_{1,3,max} > 0$, such that the scheme is A-stable provided $0<\beta\leq\beta_{1,3,max}$.		
	\end{enumerate}
	And the constant $\beta_{1,k,\max}$ for $k = 1, 2, 3$ are summarized in Table \ref{tab0}.
\end{thm}
\begin{table}[h]
	\caption{\label{tab0} $\beta_{1,k,\max}$ in Theorem \ref{thm2} for $k=1,\,2,\,3$.}
	\bigskip
	\centering
	\begin{tabular}{cccc}
		\hline
		k  &  1 &  2  & 3  \\\hline
		$\beta_{1,k,\max}$  &  2  &  1  &  1.243   \\\hline
	\end{tabular}
\end{table}

Moreover, \cite{christlieb2017kernel} showed that combining the semi-discrete scheme \eqref{eq:scheme_transport} with suitable spatial approximation to the integral \eqref{eq:I_LR}, the fully discrete scheme can be uncoditionally stable, even though the scheme is in MOL framework with explicit time discretization.
Before going further, let us give a brief intuition for why this can be achieved.  
If we simply apply Forward Euler in time and the first difference in space (FTBS scheme) on linear advection equation $u_t+c u_x=0, c>0$, we have 
\begin{align}
\label{eq:FTBS}
	\frac{u_j^{n+1}-u_j^n}{\Delta t} + c  \frac{u_j^n-u_{j-1}^n}{\Delta x}=0,
\end{align}
where $\Delta t$ is the time step, $\Delta x$ is the spatial step and $u_j^n$ is the numerical approximation to $u(x_j,t^n)$. Figurer \ref{figadd2.1} and \ref{figadd2.2} are the  stencils for this method when $c\Delta t< \Delta x$ and $c\Delta t >\Delta x$, respectively. The loss of stability in almost all explicit methods can be thought of as a lack of information to carry out the reconstruction, which is depicted in Figure \ref{figadd2.2} where the green dashed lines show the footprint of the stencil. However, different from the local method \eqref{eq:FTBS}, the kernel based approach, when combined with Forward Euler in time, is a ``global" method and has a stencil as depicted in Figurer \ref{figadd2.3}, where the green dashed lines indicate the spatial points used in the update. That is to say, for any size time step of the explicit method using the kernel based approximation to the derivative, the method has access to sufficient information.  With a few careful choices, this is validated in the analysis latter part of this paper and in our previous work.

\begin{figure}[h]
	\centering
	\subfigure[FTBS, $c\Delta t<\Delta x$] {\label{figadd2.1} \includegraphics[width=.32\textwidth]{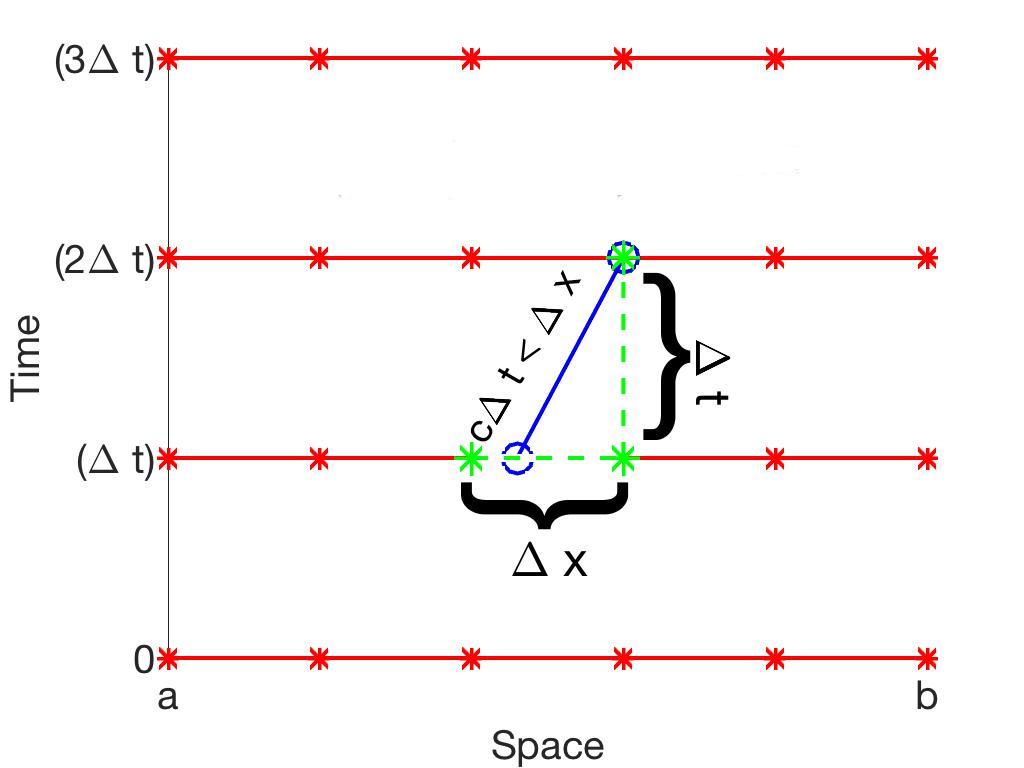}}
	\subfigure[FTBS, $c\Delta t>\Delta x$] {\label{figadd2.2} \includegraphics[width=.32\textwidth]{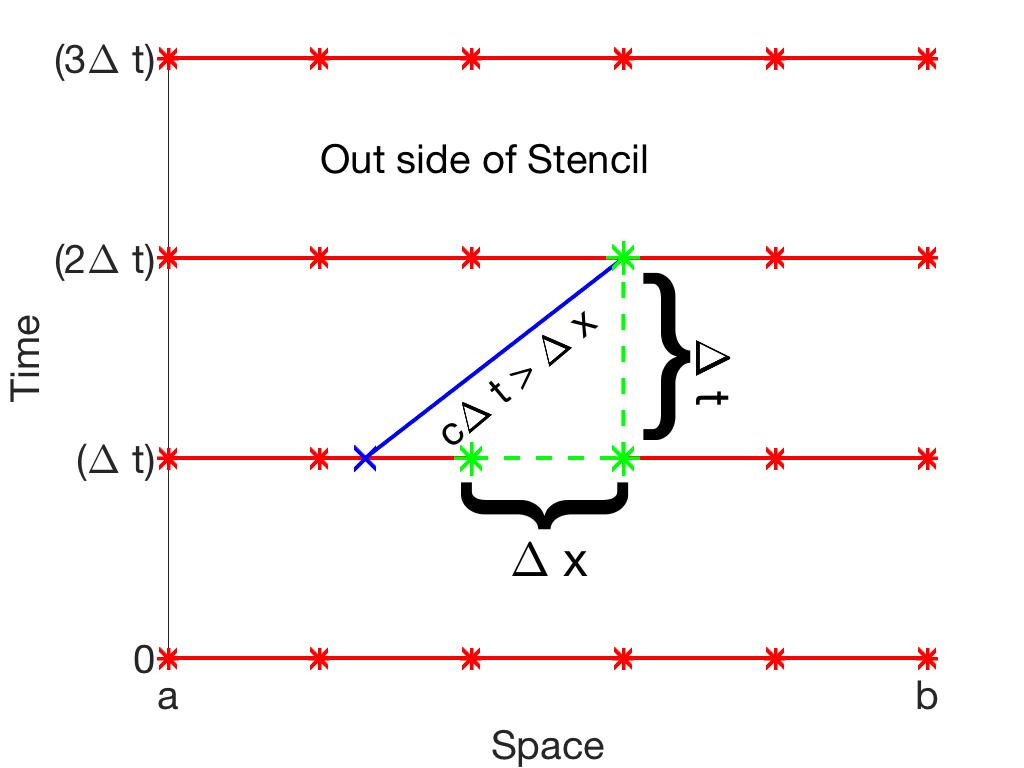}}
	\subfigure[Kernal based mathod.]{\label{figadd2.3} \includegraphics[width=.32\textwidth]{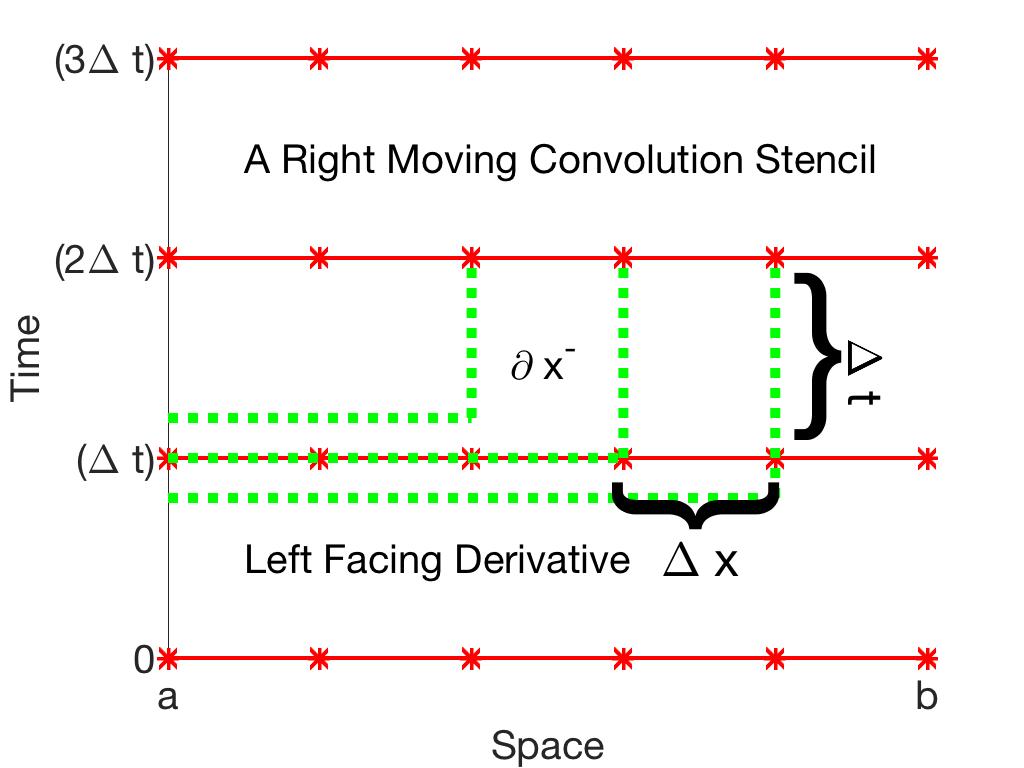}}
	\caption{Stencil used for the linear advection equation $u_t+c u_x=0$, $c>0$, with Euler forward in time.}
	\label{figadd2}
\end{figure}

Thanks to the well designed scheme for the transport part, in the following parts, we will focus on the function \eqref{eq:equation} with diffusion term only, that is
\begin{align}
\label{eq:equation2}
	\partial_t u(\mathbf{x},t) = \nabla \cdot \left( A(u, \textbf{x}, t) \nabla u\right).
\end{align}
Consequently, the system \eqref{eq:system} in one dimension turns to 
\begin{align}
\label{eq:system2}
	\left\{\begin{array}{l}
	u_t = \partial_{x} \left( A(u, x, t) w \right) ,\\
	w = \partial_{x} u. \\
	\end{array}\right.
\end{align}

\subsection{Insufficiency of the original method}

Considering the 1D system \eqref{eq:system2}, it is straightforward that we can use the partial sums \eqref{eq:partial_sum} to approach $\partial_x$. For instance,  
\begin{equation} 
\label{eq:newmth1}
\begin{aligned}
	& (A w)_x \approx \alpha \sum\limits_{p=1}^{2k}\mathcal{D}_L^p[Aw,\alpha] 
	=:\mathcal{H}^k_1[u,A,\alpha], \\
	& w=u_{x} \approx -\alpha\sum\limits_{q=1}^{2k}\mathcal{D}_R^q[u,\alpha].
\end{aligned}
\end{equation}

\noindent
The operator $\mathcal{H}^k_1[u,A,\alpha]$ has a truncation error $\mathcal{O}(1/\alpha^{2k})$. Considering the linear function $u_t=A u_{xx}$ with $A>0$, the scheme that employs \eqref{eq:newmth1} and $k$-th order SSP RK method can be proved to be A-stable, if we take $\alpha=\sqrt{\beta/(A\Delta t)}$ and $0<\beta\leq \beta_{k}$.
However, we found that the A-stable interval $(0, \beta_{k})$ was pretty narrow, e.g., $\beta_{3}=2/9$. As a consequence, schemes general large error $1/\alpha^{2k} = (A\Delta t / \beta)^k \geq (A\Delta t/\beta_k)^k$.  

Besides that, the schemes show another disadvantage.
As an example, we look at a special case that $u=\sin x$, $A(u,x)=1$ and the interval $[a,b]=[0,2\pi]$. Then approximation \eqref{eq:newmth1} would be
	$$\mathcal{H}^k_1[u,A,\alpha](x) = - \left( \alpha \sum\limits_{p=1}^{2k} \mathcal{D}_L^p \right) \star \left( \alpha\sum\limits_{q=1}^{2k}\mathcal{D}_R^q\right)[\sin x, \alpha]. $$ 
	In Figure \ref{figerr1}, we plot the error $\Vert (\sin x)_{xx} - \mathcal{H}^k_1[\sin x,1,\alpha] \Vert_\infty$ for $k=1,2,3$. It is observed that the error is not a monotone decreasing function of $\alpha$, indicating that the scheme cannot converge uniformly. 
	Consequently, refining meshes, or increasing $\alpha$ equivalently, could result in larger error. Furthermore, changing the function to $u=\sin(2x)$, the error lines in Figure \ref{figerr1ex} tell us that the monotone decrasing interval of each scheme would change at the same time. This means for each scheme we can not find a uniform monotone interval for all smooth functions.
	{\color{red}}

\begin{figure}[htb]
\centering
		\subfigure[$k=1$]{\includegraphics[width=.32\textwidth]{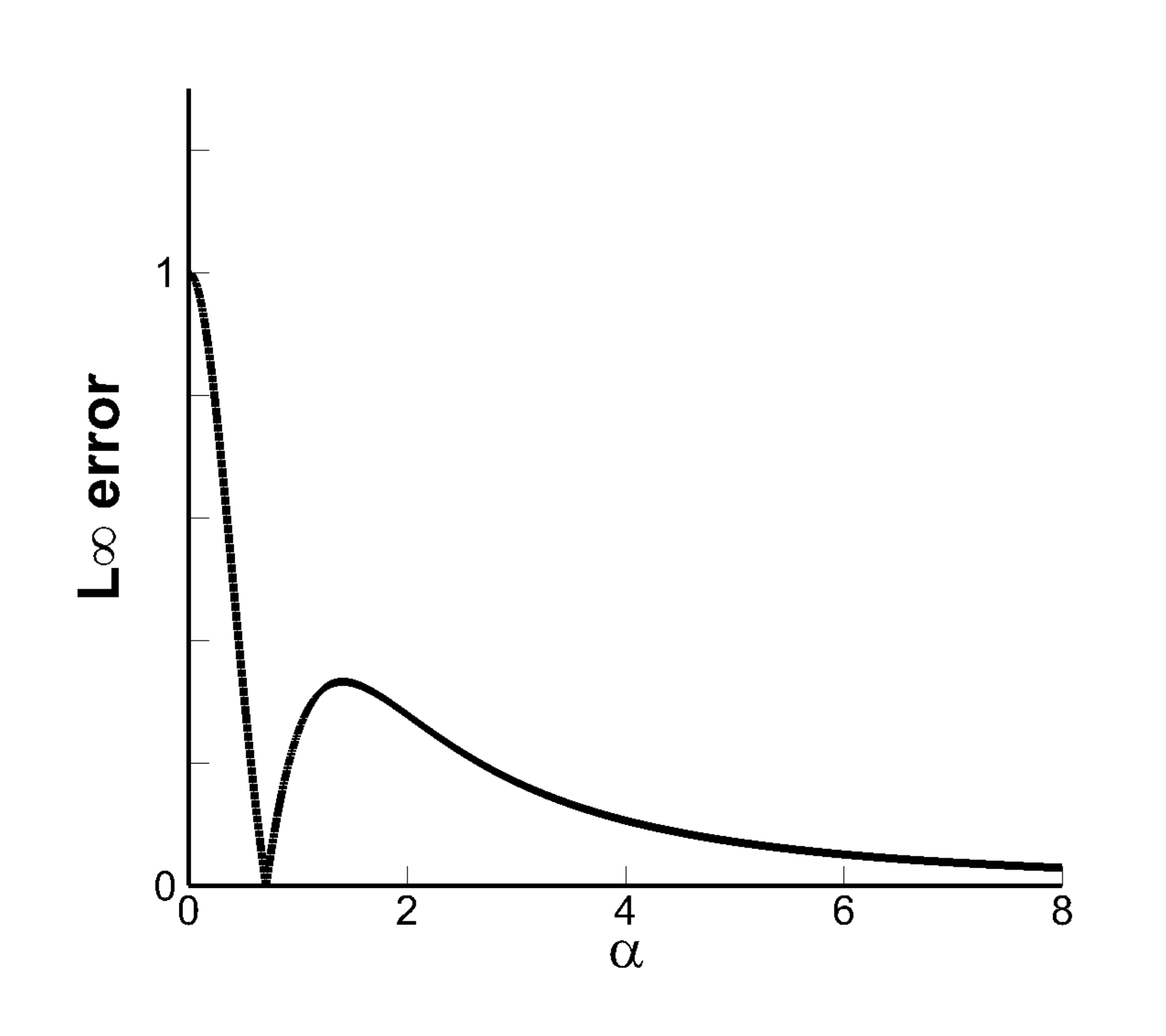}}
		\subfigure[$k=2$]{\includegraphics[width=.32\textwidth]{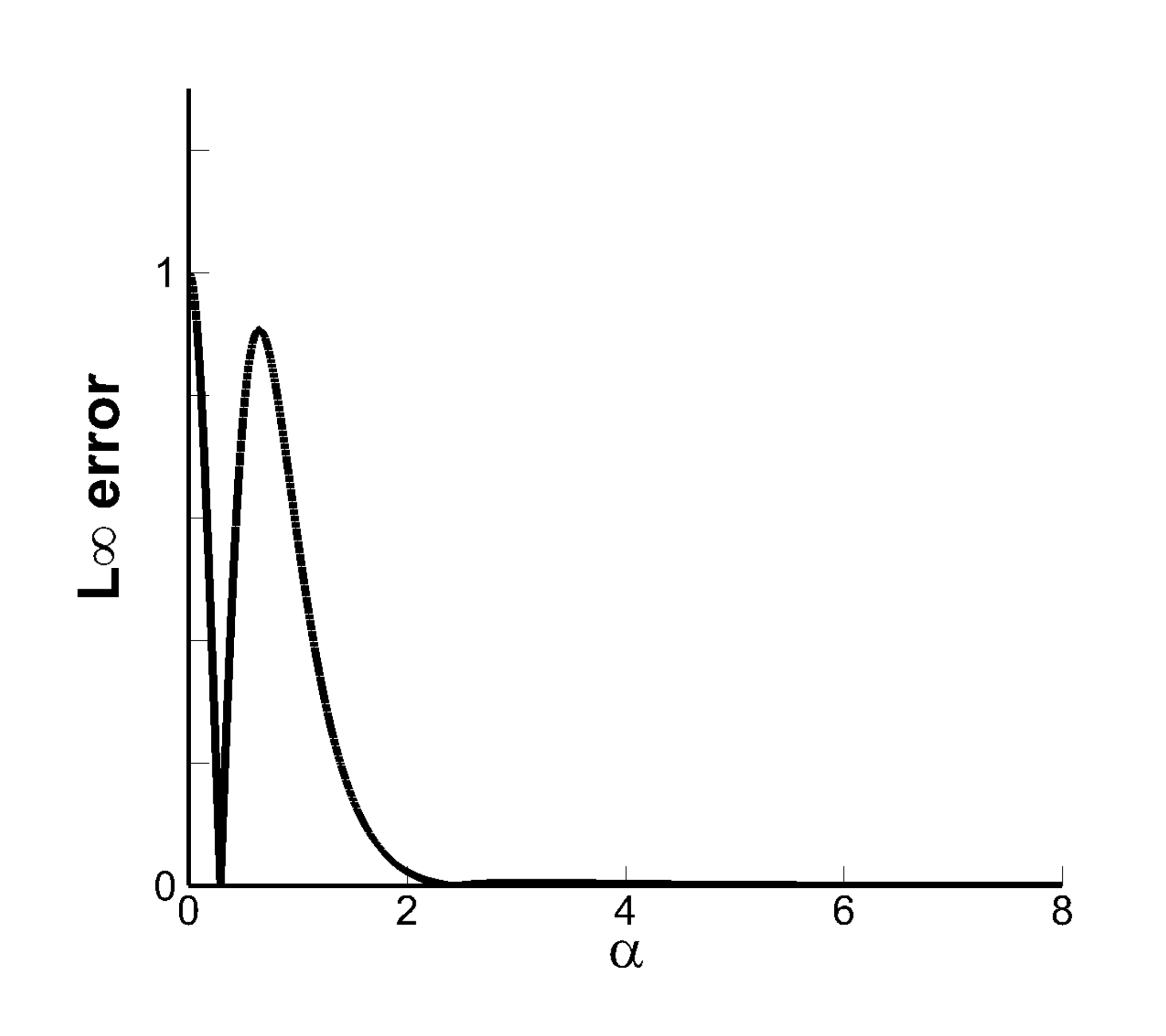}}
		\subfigure[$k=3$]{\includegraphics[width=.32\textwidth]{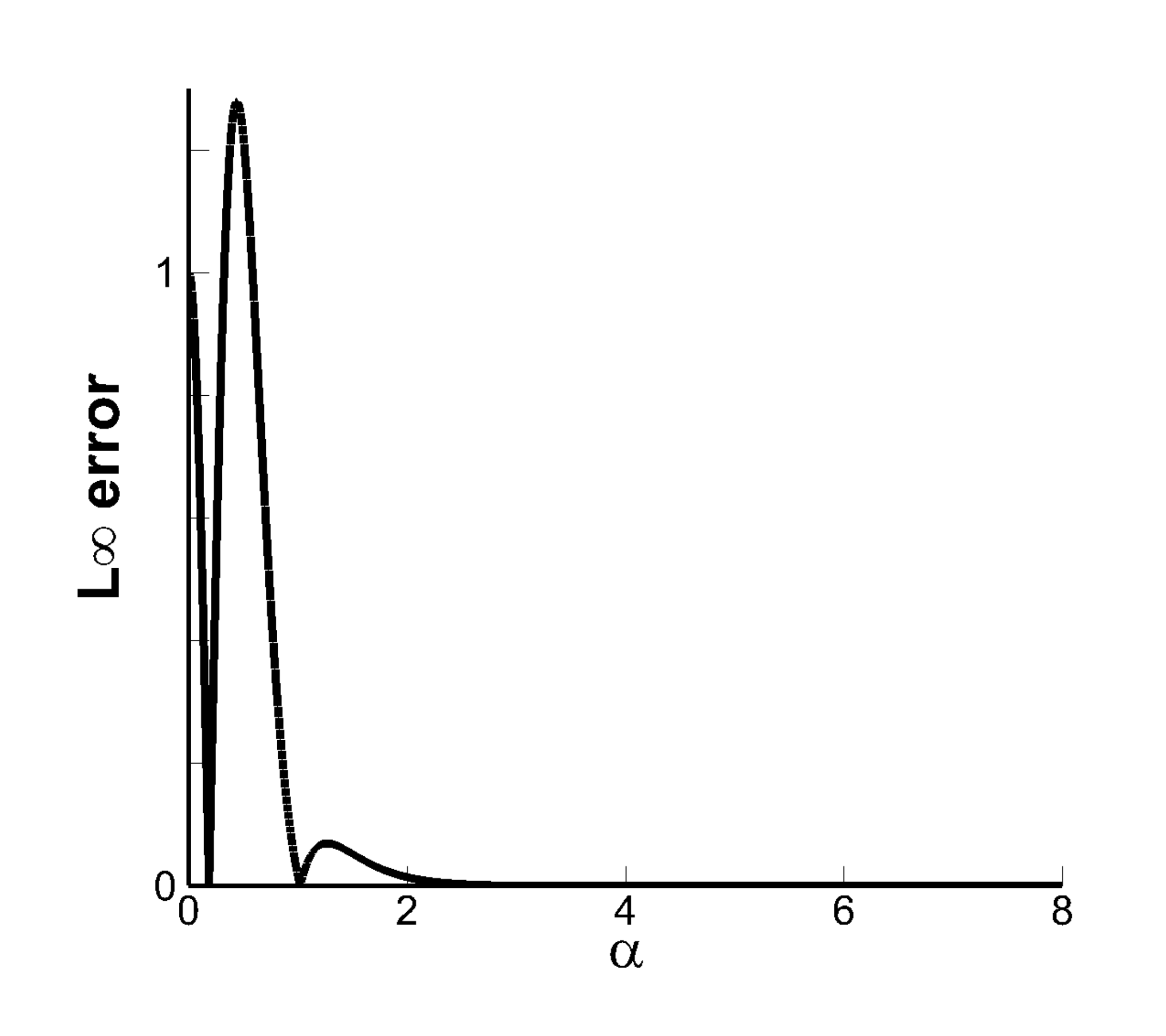}}
	\caption{$\Vert (\sin x)_{xx}-\mathcal{H}^k_1[\sin x,1,\alpha] \Vert_\infty$, with $\mathcal{H}^k_1$ given in \eqref{eq:newmth1}.}
\label{figerr1}
\end{figure}

\begin{figure}[htb]
	\centering
	\subfigure[$k=1$]{\includegraphics[width=.32\textwidth]{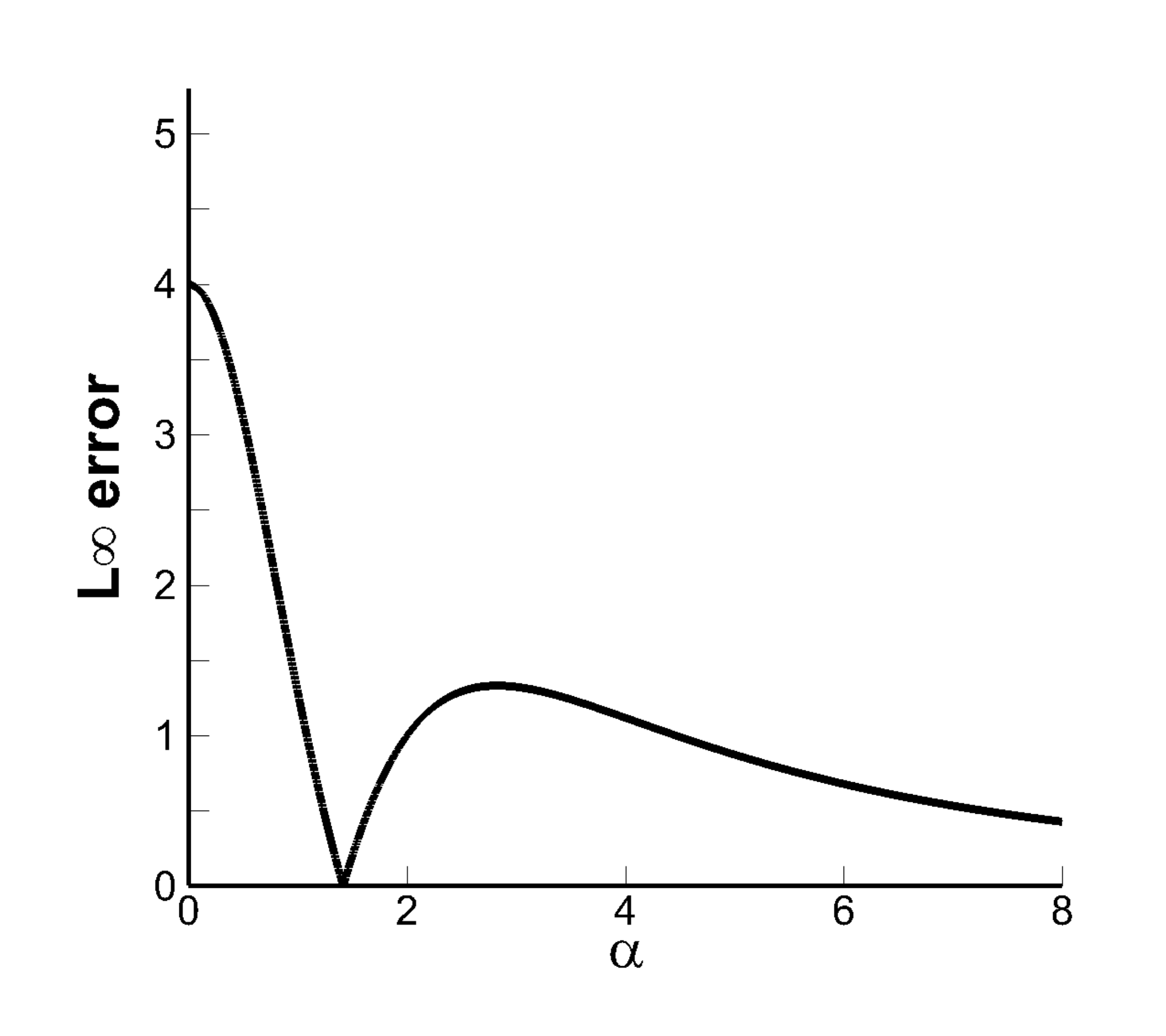}}
	\subfigure[$k=2$]{\includegraphics[width=.32\textwidth]{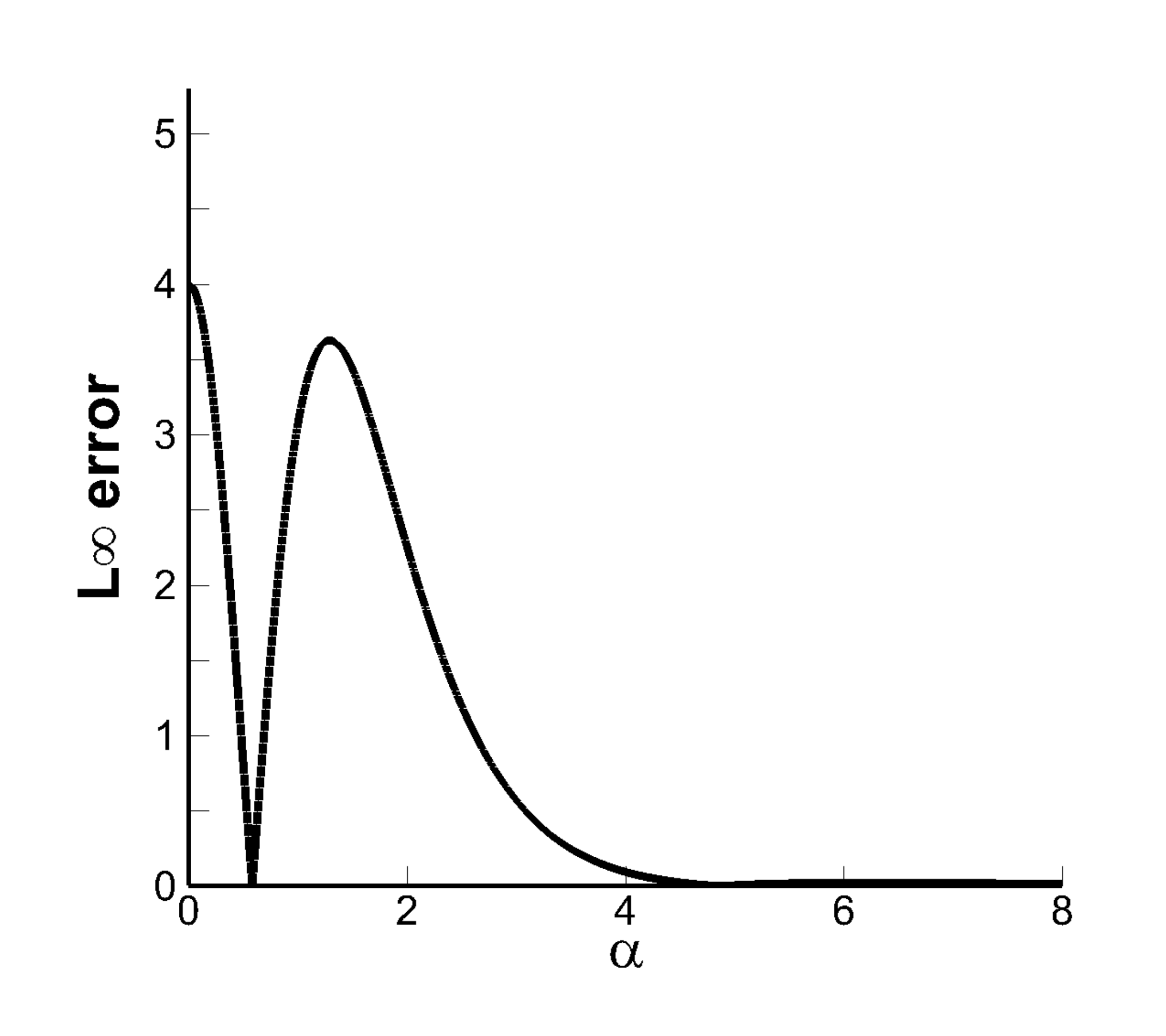}}
	\subfigure[$k=3$]{\includegraphics[width=.32\textwidth]{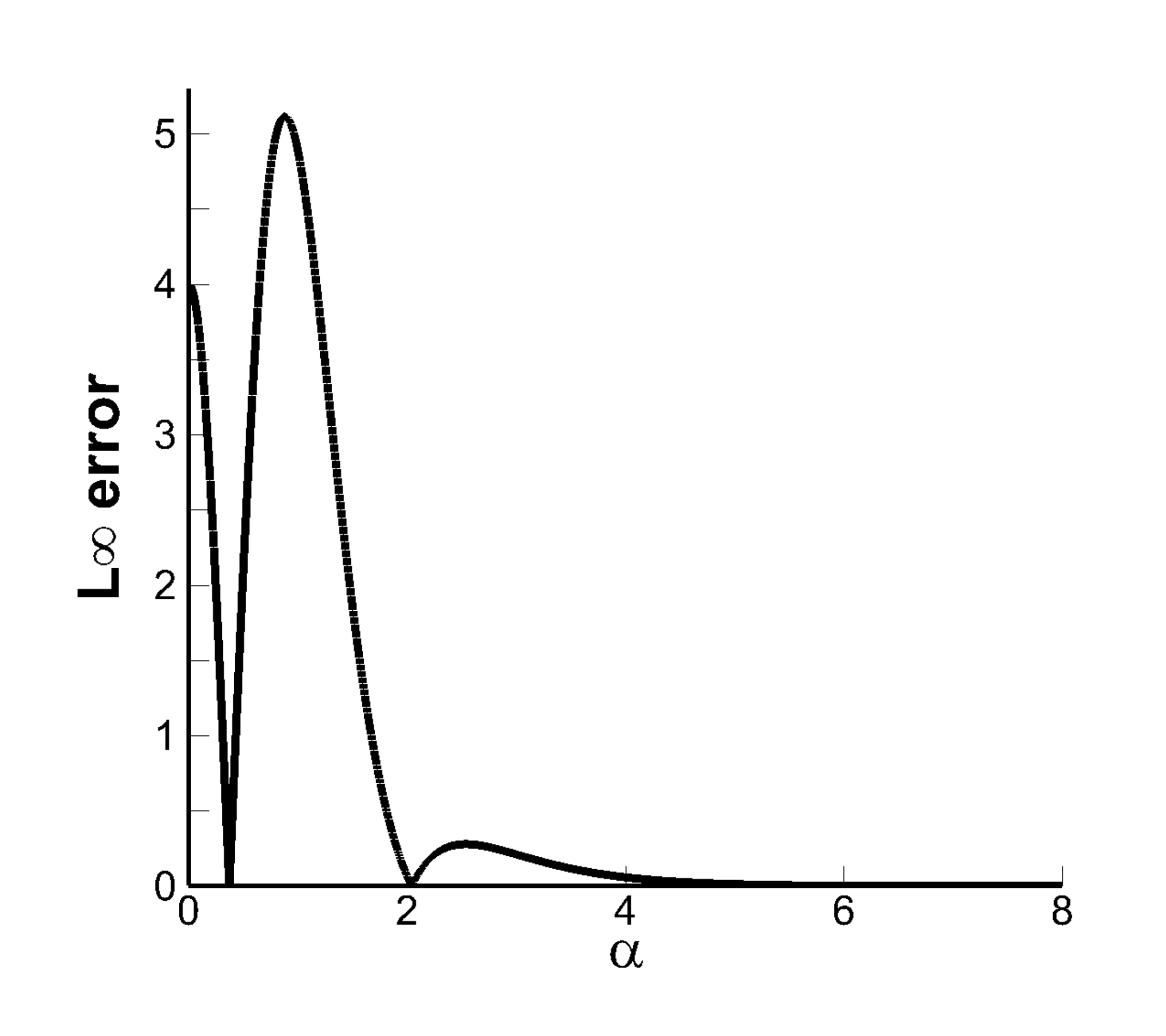}}
	\caption{$\Vert (\sin 2x)_{xx}-\mathcal{H}^k_1[\sin 2x,1,\alpha] \Vert_\infty$, with $\mathcal{H}^k_1$ given in \eqref{eq:newmth1}.}
	\label{figerr1ex}
\end{figure}


Alternatively, it is naturally to use the average $\frac{1}{2} \alpha \sum\limits_{p=1}^{2k} (\mathcal{D}_L^p-\mathcal{D}_R^p)$ to approach $\partial_x$ in \eqref{eq:system2}, that is
\begin{equation}
\label{eq:method2}
\begin{aligned}
	& (A w)_x \approx \frac{\alpha}{2} \sum\limits_{p=1}^{2k} \left( \mathcal{D}_L^p[Aw,\alpha] -\mathcal{D}_R^p[Aw,\alpha] \right)
	=:\mathcal{H}^k_2[u,A,\alpha] \\
	& w=u_{x} \approx \frac{\alpha}{2} \sum\limits_{q=1}^{2k}\left( \mathcal{D}_L^q[u,\alpha] - \mathcal{D}_R^q[u,\alpha] \right).
\end{aligned}
\end{equation}
The scheme has a truncation error $\mathcal{O}(1/\alpha^{2k})$ and the same problem as scheme \eqref{eq:newmth1}, learning from Figure \ref{figerr2} with
	$$\mathcal{H}^k_2[ \sin x,1,\alpha](x) = \left(\frac{\alpha}{2} \sum\limits_{p=1}^{2k} \left( \mathcal{D}_L^p -\mathcal{D}_R^p \right)\right) \star \left(\frac{\alpha}{2} \sum\limits_{q=1}^{2k}\left( \mathcal{D}_L^q - \mathcal{D}_R^q \right)\right)[\sin x, \alpha]. $$ 


\begin{figure}[htb]
\centering
	\subfigure[$k=1$]{\includegraphics[width=.32\textwidth]{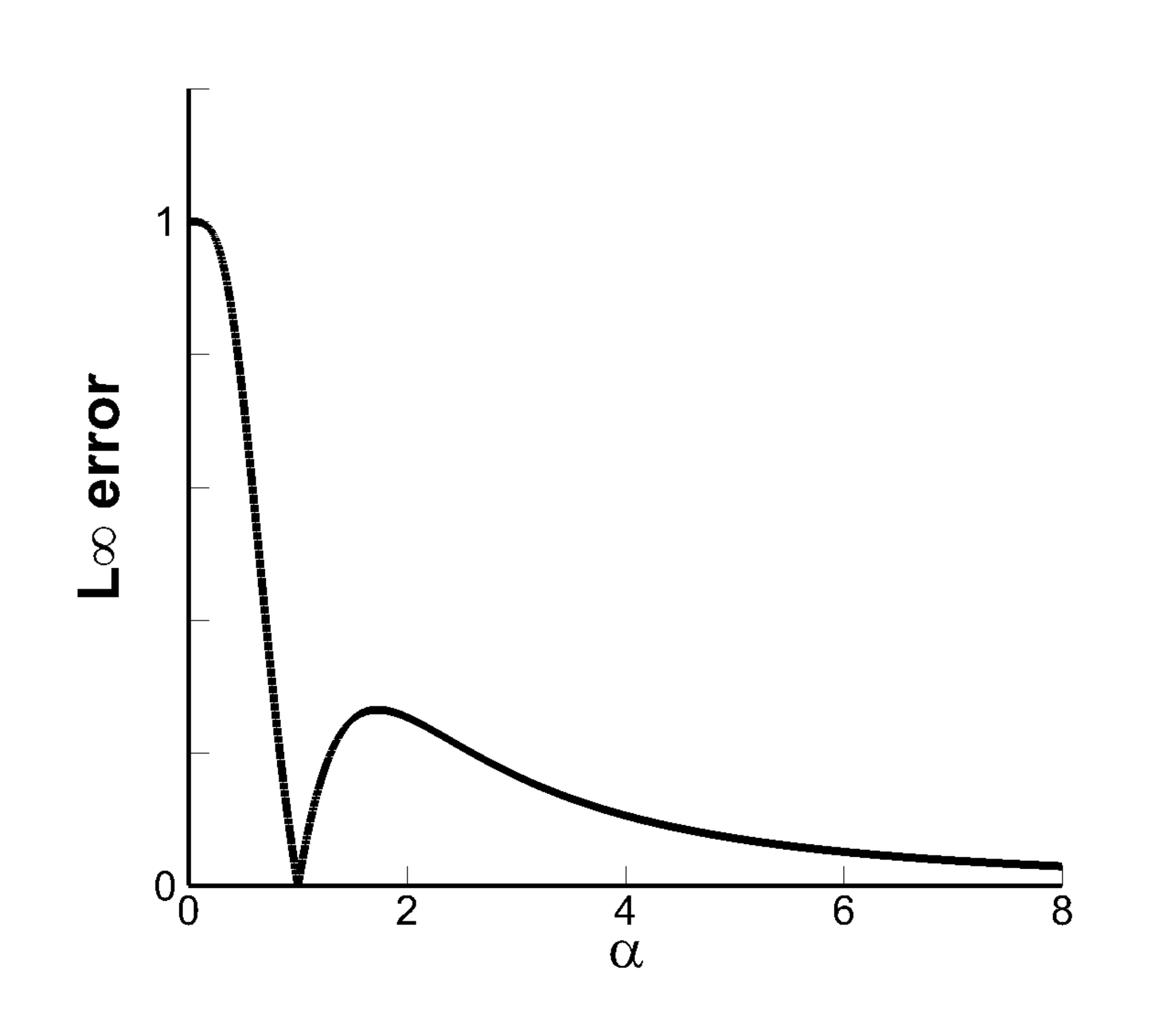}}
	\subfigure[$k=2$]{\includegraphics[width=.32\textwidth]{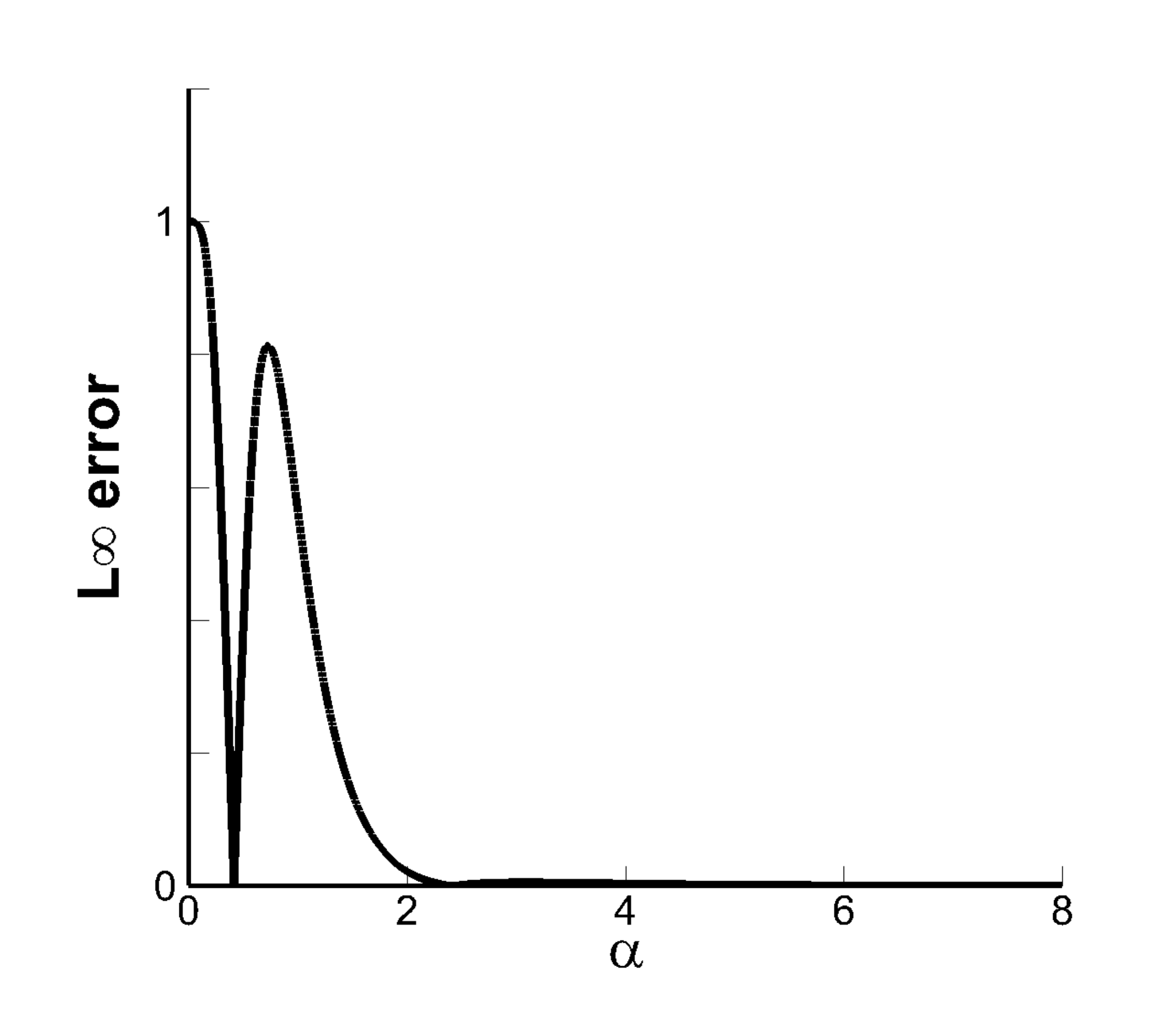}}
	\subfigure[$k=3$]{\includegraphics[width=.32\textwidth]{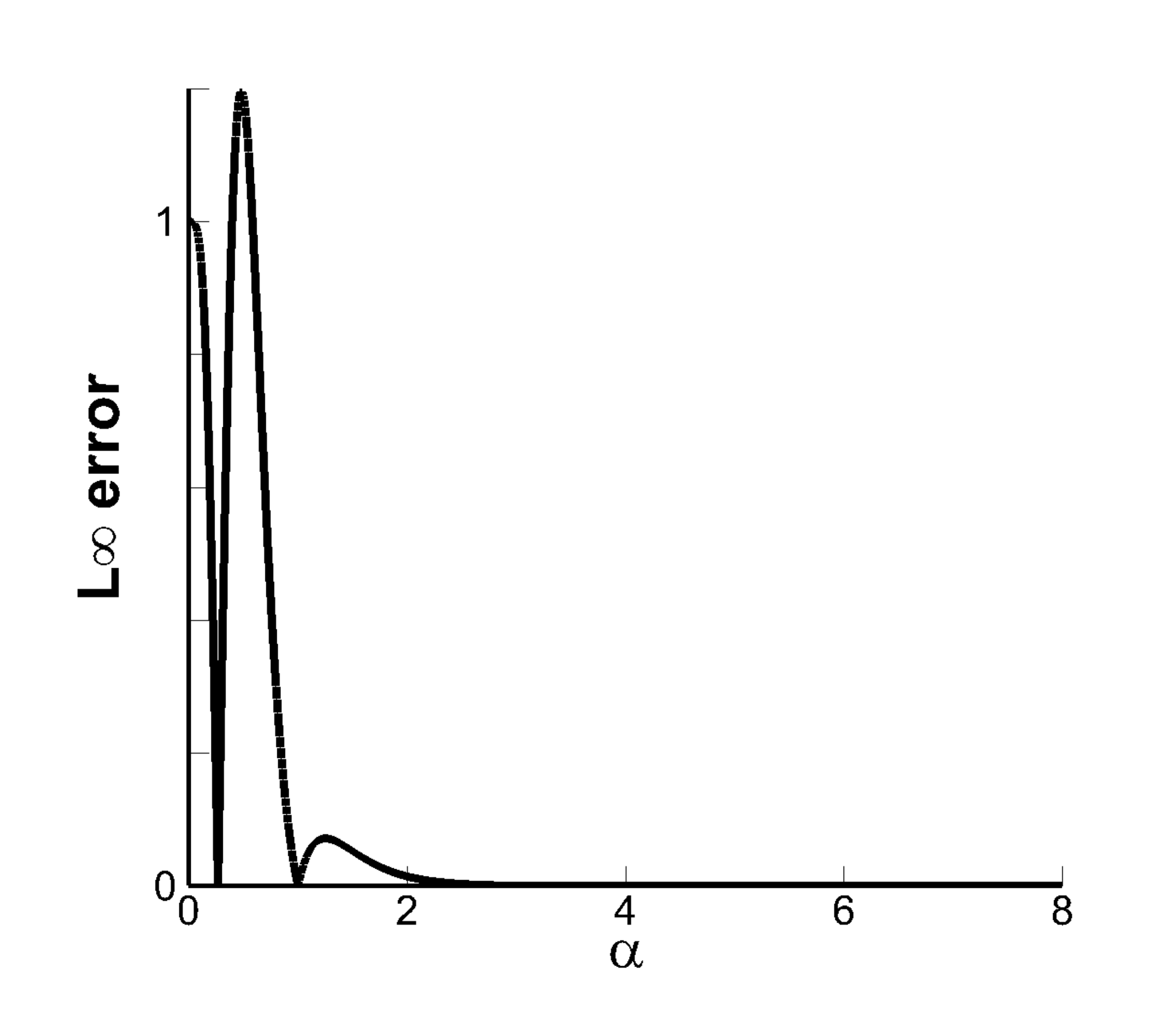}}
	\caption{$\Vert (\sin x)_{xx}-\mathcal{H}^k_2[\sin x,1,\alpha] \Vert_\infty$, with $\mathcal{H}^k_2$ given in \eqref{eq:method2}.}
\label{figerr2}
\end{figure}

\begin{figure}[htb]
	\centering
	\subfigure[$k=1$]{\includegraphics[width=.32\textwidth]{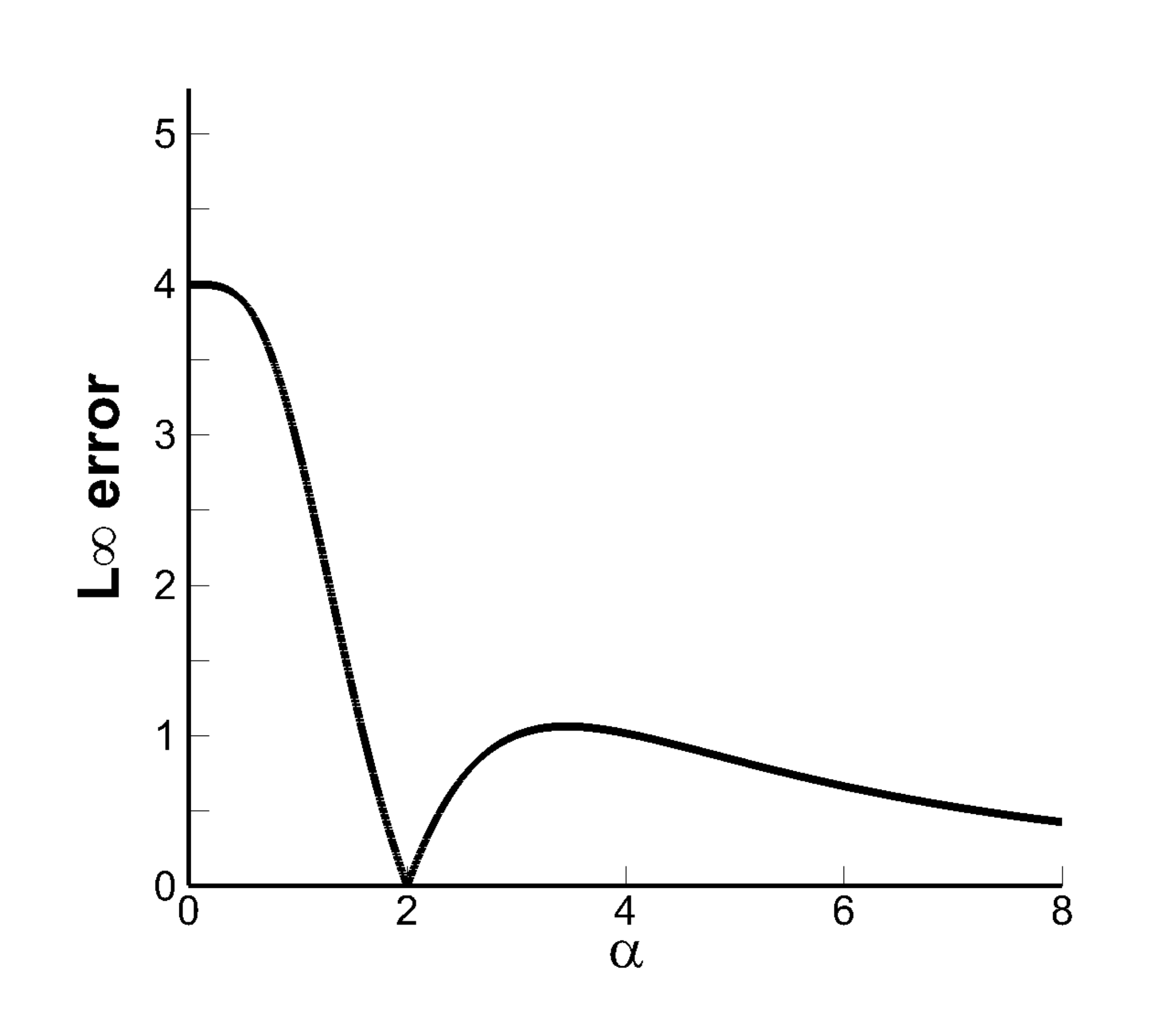}}
	\subfigure[$k=2$]{\includegraphics[width=.32\textwidth]{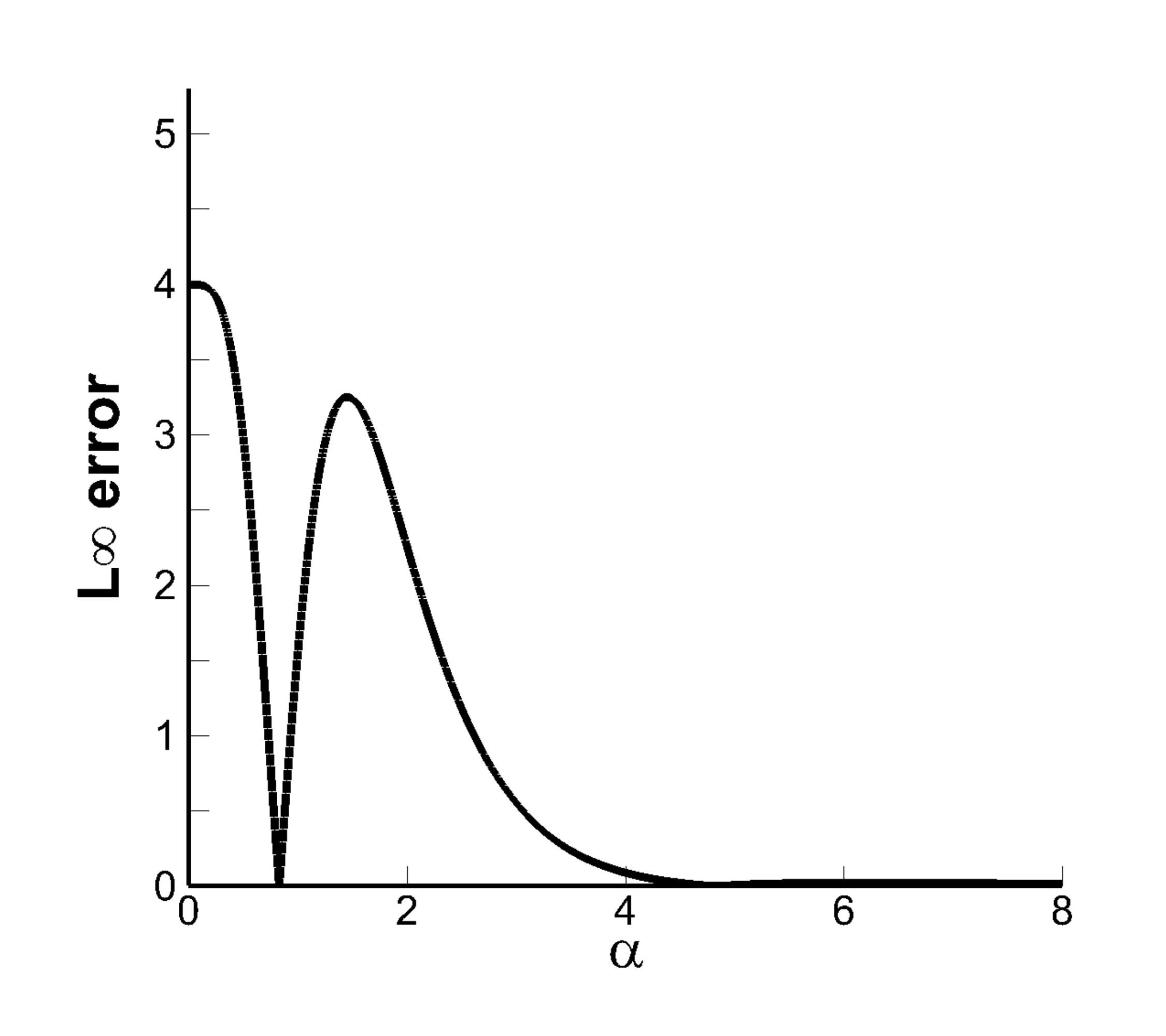}}
	\subfigure[$k=3$]{\includegraphics[width=.32\textwidth]{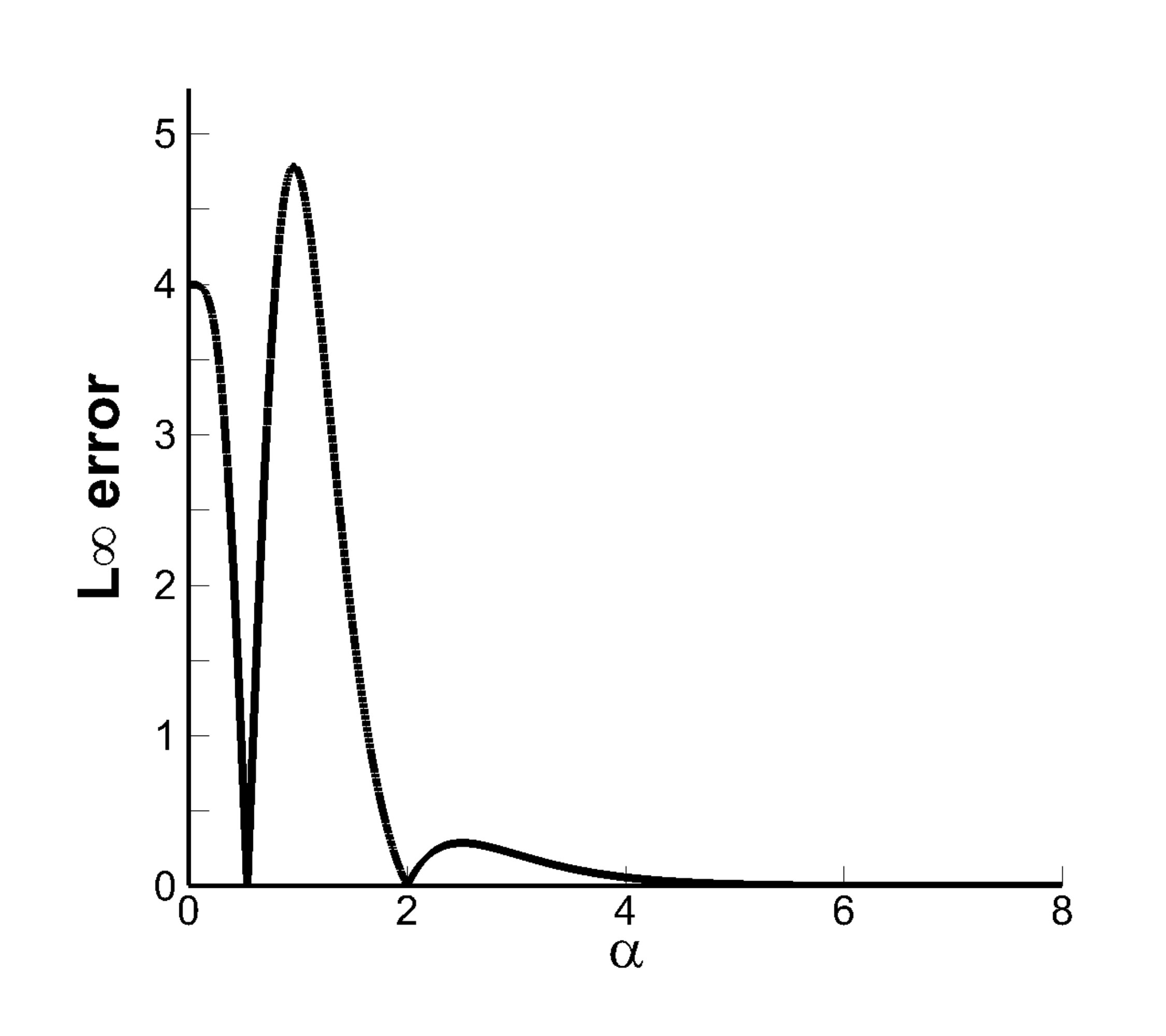}}
	\caption{$\Vert (\sin 2x)_{xx}-\mathcal{H}^k_2[\sin 2x,1,\alpha] \Vert_\infty$, with $\mathcal{H}^k_2$ given in \eqref{eq:method2}.}
	\label{figerr2ex}
\end{figure}
\section{New Representation of Differential Operators}
In this section, we will introduce a new representation of the first order differential operator $\partial_{x}$, and further use it in \eqref{eq:system2}. We require that the proposed scheme can maintain the high order accuracy and the A-stable property. Moreover, it can overcome the disadvantages of the schemes \eqref{eq:newmth1} and \eqref{eq:method2}, so that the error converges uniformly for $\alpha$, and the A-stable interval is relatively larger. 

\subsection{Construction of A New Representation}

We have showed in last section that the scheme \eqref{eq:method2} with $2k=2$ has a truncation error $\mO(1/\alpha^2)$. In fact, based on Theorem \ref{thm1}, we can obtain that
\begin{align}
\label{eq:dxx1}
	 \frac{\alpha}{2}(\mathcal{D}_L-\mathcal{D}_R)[v,\alpha](x) 
	=& \sum\limits_{p=1}^m\left(\frac{1}{\alpha}\right)^{2p-2}\partial_x^{2p-1}v(x) - \frac{1}{2}\left(\frac{1}{\alpha}\right)^{2m}(\mathcal{L}_L^{-1}-\mathcal{L}_R^{-1})[\partial_x^{2m+1}v,\alpha](x) \notag \\
	=& \partial_{x} v(x) + \left(\frac{1}{\alpha}\right)^{2}\partial_x^{3}v(x) - \frac{1}{2}\left(\frac{1}{\alpha}\right)^{4}(\mathcal{L}_L^{-1}-\mathcal{L}_R^{-1})[\partial_x^{5}v,\alpha](x).
\end{align}
This demonstrates that the operator $\frac{\alpha}{2}(\mathcal{D}_L-\mathcal{D}_R)$,  i.e., \eqref{eq:method2} with $2k=1$, approximates the first order derivative with error $\mO(1/\alpha^2)$ as well. Therefore, we can achieve the same order of accuracy with less computational cost. If we define $\mH^1_3[u,A,\alpha](x)$ as 
\begin{align}
	\mH^1_3[u,A,\alpha](x) =\frac{\alpha}{2}(\mathcal{D}_L-\mathcal{D}_R)[Aw,\alpha](x), 
	\quad \text{and} \quad
	w=\frac{\alpha}{2}(\mathcal{D}_L-\mathcal{D}_R)[u,\alpha](x). 
\end{align}
Then, the error $\Vert (\sin x)_{xx}-\mathcal{H}^1_3[\sin x,1,\alpha] \Vert_\infty$ is a monotone decreasing function of $\alpha$,  see Figure \ref{fig:err3_k1}.
Hence, we will start from \eqref{eq:dxx1} and construct new higher order approximations of $\partial_x$, by ``removing" the higher order derivatives. 

To eliminate the main error term in \eqref{eq:dxx1}, i.e., $(1/\alpha)^2 \partial_x^3 v(x)$, we introduce another operation $\mD_0$ here,
 \begin{align}
 \label{eq:D0}
	\mathcal{D}_0[v,\alpha](x) 
	:=& \frac{1}{2} \left( \mD_{L} + \mD_{R} \right)[v,\alpha](x) = (\mI -\mL_{0}^{-1})[v,\alpha](x) \notag\\
	=& -\sum\limits_{p=1}^k \left(\frac{1}{\alpha}\right)^{2p} \partial_x^{2p}v(x) -\left( \frac{1}{\alpha} \right)^{2k+2} \mathcal{L}_0^{-1}[\partial_x^{2k+2}v,\alpha](x),
 \end{align} 
where, 
 \begin{align*}
 	\mL_{0}^{-1}[v,\alpha](x) =  \frac{\alpha}{2} \int_a^b e^{-\alpha |x-y|}v(y)dy + A_0 e^{-\alpha(x-a)} + B_0 e^{-\alpha(b-x)}
 \end{align*}
 and the coefficients $A_0=\frac{I^0[v,\alpha](b)}{1-\mu}$ and $B_0=\frac{I^0[v,\alpha](a)}{1-\mu}$ for periodic boundary conditions. The last equality in \eqref{eq:D0} is given in \cite{christlieb2017kernel}.
 Consequently, we can easily prove that
\begin{align*}
	 \mathcal{D}_0 \left[ \frac{\alpha}{2}(\mathcal{D}_L-\mathcal{D}_R) \right] + \frac{\alpha}{2}(\mathcal{D}_L-\mathcal{D}_R) 
	 = \frac{\alpha}{2} \left( \mI + \mD_{0} \right) \left( \mD_{L} -\mD_{R} \right) 
	 = \partial_x - \left(\frac{1}{\alpha}\right)^4\partial_x^5 + \mathcal{O} \left( \frac{1}{\alpha^6} \right),
\end{align*} 
reducing error from $\mO(1/\alpha^2)$ to $\mO(1/\alpha^4)$.
Repeat this process to eliminate the higher order derivatives in turn, and we can have a general form of the scheme with error $\mO(1/\alpha^{2k})$, which is showed in the following theorem.

\begin{thm}
\label{th:errd0dldr}
	Suppose $v(x)\in C^{2k+1}([a,b])$ is a periodic function. Consider the operators $\mD_{*}$ with the periodic boundary treatment $\mD_{*}(a) = \mD_{*}(b)$, where $*$ can be 0, L or R. Then, we have
	\begin{align}
		\Vert \partial_xv(x) - \frac{\alpha}{2} \sum\limits_{p=1}^k\mathcal{D}_0^{p-1} (\mathcal{D}_L-\mathcal{D}_R)[v,\alpha](x) \Vert_\infty \leq C \left( \frac{1}{\alpha} \right)^{2k} \Vert\partial_x^{2k+1}v\Vert_\infty,
	\end{align}
	where $C$ is a constant only depending on k.
\end{thm}

\begin{proof}
Using the definition of $\mL^{-1}_{L}$, it is easy to deduce that
\begin{align*}
	\partial_x\mathcal{L}^{-1}_L[v,\alpha](x)
	&=\alpha v(x)-\alpha^2\int_a^x v(y) e^{-\alpha(x-y)}dy-\alpha e^{-\alpha(x-a)} A_L[v,\alpha]\\
	&=\alpha v(x)-\alpha\mathcal{L}^{-1}_L[v,\alpha](x)=\alpha\mathcal{D}_L[v,\alpha](x).
\end{align*}
Hence,
\begin{align*}
	\partial_x\mathcal{D}_L[v,\alpha](x) 
	=\partial_x v(x)-\alpha\mathcal{D}_L[v,\alpha](x).
\end{align*}
On the other hand, using integration by parts, we can obtain that 
\begin{align*}
	\mathcal{D}_L[v,\alpha](x)
	&= v(x) - \alpha \int_{a}^{x} e^{-\alpha(x-y)} v (y) dy - \frac{I^L[v,\alpha](b)}{1-\mu} e^{-\alpha(x-a)} \\
	&= v(x) - e^{-\alpha(x-y)} v(y) |_{y=a} ^{y=x} + \int_{a}^{x} e^{-\alpha(x-y)} v'(y) dy -\frac{v(b)-\mu v(a) -\frac{1}{\alpha} I^{L}[v',\alpha](b)}{1-\mu}e^{-\alpha(x-a)} \\
	&=\frac{1}{\alpha}\mathcal{L}^{-1}_L[\partial_x v,\alpha](x)
=\frac{1}{\alpha}\partial_xv(x)-\frac{1}{\alpha}\mathcal{D}_L[\partial_xv,\alpha](x),
\end{align*}
	or equivalently,
\begin{align*}
\mathcal{D}_L[\partial_x v,\alpha]=\partial_x v(x)-\alpha \mathcal{D}_L[v,\alpha](x).
\end{align*}
Therefore, 
$$\partial_x\mathcal{D}_L[v,\alpha]=\mathcal{D}_L[\partial_x v,\alpha].$$
	Similarly, $\partial_x\mathcal{D}_R[v,\alpha]=\mathcal{D}_R[\partial_x v,\alpha]$ and $\partial_{xx}\mathcal{D}_0[v,\alpha]=\mathcal{D}_0[\partial_{xx} v,\alpha]$. 


Next, let we consider operator $\frac{1}{2}\mathcal{D}_0(\mathcal{D}_L-\mathcal{D}_R)$. Note that there is a general form of \eqref{eq:D0} for $0\leq p<k$:
\begin{align*}
\mathcal{D}_0[\partial_x^{2p-1}v,\alpha](x)&=-\sum\limits_{m=p+1}^{k}\left(\frac{1}{\alpha}\right)^{2(m-p)}\partial_x^{2m-1}v(x)-\left(\frac{1}{\alpha}\right)^{2(k-p+1)}\mathcal{L}_0^{-1}[\partial_x^{2k+1}v,\alpha](x).
\end{align*}
Hence, we can obtain that
\begin{align*}
\frac{1}{2}\mathcal{D}_0(\mathcal{D}_L-\mathcal{D}_R)=&\sum\limits_{p=1}^{k-1}\left(\frac{1}{\alpha}\right)^{2p-1}\mathcal{D}_0[\partial_x^{2p-1}v,\alpha](x)-\frac{1}{2}\left(\frac{1}{\alpha}\right)^{2k-1}\mathcal{D}_0[(\mathcal{L}_L^{-1}-\mathcal{L}_R^{-1})[\partial_x^{2k-1}v,\alpha],\alpha](x)\\
=&\sum\limits_{p=1}^{k-1}\left(\frac{1}{\alpha}\right)^{2p-1}\left(-\sum\limits_{m=p+1}^{k}\left(\frac{1}{\alpha}\right)^{2(m-p)}\partial_x^{2m-1}v(x)-\left(\frac{1}{\alpha}\right)^{2(k-p+1)}\mathcal{L}_0^{-1}[\partial_x^{2k+1}v,\alpha](x)\right)\\
&+\frac{1}{2}\left(\frac{1}{\alpha}\right)^{2k+1}\mathcal{L}_0^{-1}(\mathcal{L}_L^{-1}-\mathcal{L}_R^{-1})[\partial_x^{2k+1}v,\alpha](x)\\
=-&\sum\limits_{p=2}^k(p-1)\left(\frac{1}{\alpha}\right)^{2p-1}\partial_x^{2p-1}v(x)-\frac{k-1}{2}\left(\frac{1}{\alpha}\right)^{2k+1}(\mathcal{L}_L^{-1}-\mathcal{L}_R^{-1})[\partial_x^{2k+1}v,\alpha](x)\\
&+\frac{1}{2}\left(\frac{1}{\alpha}\right)^{2k+1}\mathcal{L}_0^{-1}(\mathcal{L}_L^{-1}-\mathcal{L}_R^{-1})[\partial_x^{2k+1}v,\alpha](x).
\end{align*}
Therefore,
$$
\frac{1}{2}(\mathcal{D}_L-\mathcal{D}_R)+\frac{1}{2}\mathcal{D}_0(\mathcal{D}_L-\mathcal{D}_R)=\frac{1}{\alpha}v_x(x)-\sum\limits_{p=3}^k(p-2)\left(\frac{1}{\alpha}\right)^{2p-1}\partial^{2p-1}_xv(x)-\left(\frac{1}{\alpha}\right)^{2k+1}Q_2[v,\alpha](x),
$$
where $Q_2[v,\alpha](x)=\frac{k-1}{2}(\mathcal{L}_L^{-1}-\mathcal{L}_R^{-1})[\partial_x^{2k-1}v,\alpha](x)-\frac{1}{2}\mathcal{L}_0^{-1}(\mathcal{L}_L^{-1}-\mathcal{L}_R^{-1})[\partial_x^{2k+1}v,\alpha](x)]$.
Repeat the process, and we finally obtain 
$$
\frac{1}{2}\sum\limits_{p=1}^k\mathcal{D}_0^{p-1}(\mathcal{D}_L-\mathcal{D}_R)[v,\alpha](x)=\frac{1}{\alpha}\partial_xv(x)+\frac{(-1)^k}{2}\left(\frac{1}{\alpha}\right)^{2k+1}(\mathcal{L}_0^{-1})^{k-1}(\mathcal{L}_L^{-1}-\mathcal{L}_R^{-1})[\partial_x^{2k+1}v,\alpha](x).
$$ 
Note that for any $w(x)\in C([a,b])$, we can find a constant $\bar{C}$ independent of $w$ and $\alpha$, such that $\|\mathcal{L}_{*}^{-1}[w,\alpha](x)\|_\infty\leq \bar{C}\|w\|_\infty$, where $*$ can be $0$, $L$ and $R$.
Hence, there is a constant $C$ that only depends on $k$ satisfying
$$
\left \Vert \frac{1}{\alpha} \partial_xv(x) - \frac{1}{2} \sum\limits_{p=1}^k \mathcal{D}_0^{p-1}(\mathcal{D}_L-\mathcal{D}_R)[v,\alpha](x) \right\Vert_\infty \leq  C \left( \frac{1}{\alpha} \right)^{2k+1}\Vert\partial_x^{2k+1}v\Vert_\infty.
$$
And the theorem is proved. 
\end{proof}

Then we get the novel approximation of $(A(u,x) u_{x})_{x}$,
\begin{equation}
\label{eq:1D}
\begin{aligned}
	& (A(u,x)u_x)_x \approx \frac{\alpha}{2} \sum\limits_{p=1}^k\mathcal{D}_0^{p-1} (\mathcal{D}_L-\mathcal{D}_R)[Aw,\alpha](x) =: \mathcal{H}_{3}^{k}[u,A, \alpha](x), \\
	& w=u_x \approx \frac{\alpha}{2} \sum\limits_{p=1}^k\mathcal{D}_0^{p-1} (\mathcal{D}_L-\mathcal{D}_R)[v,\alpha](x),
\end{aligned}
\end{equation}

\noindent
and the $L_\infty$ error has the order of $(1/\alpha)^{2k}$.
Again, we test this operation $\mathcal{H}_{3}^{k}$ with the specific case that $u=\sin x$, $A=1$ and $[a,b]=[0,2\pi]$, and concern about the monotonicity of the $L_\infty$ errors with respect to $\alpha$. Figure \ref{figerr3} indicates the monotonicity and uniform convergence of the novel scheme. 
And for the function $u=\sin(2x)$, we have the similar conclusion, which will not be presented any more.

\begin{figure}[htb]
\centering
	\subfigure[$k=1$]{\includegraphics[width=.3\textwidth]{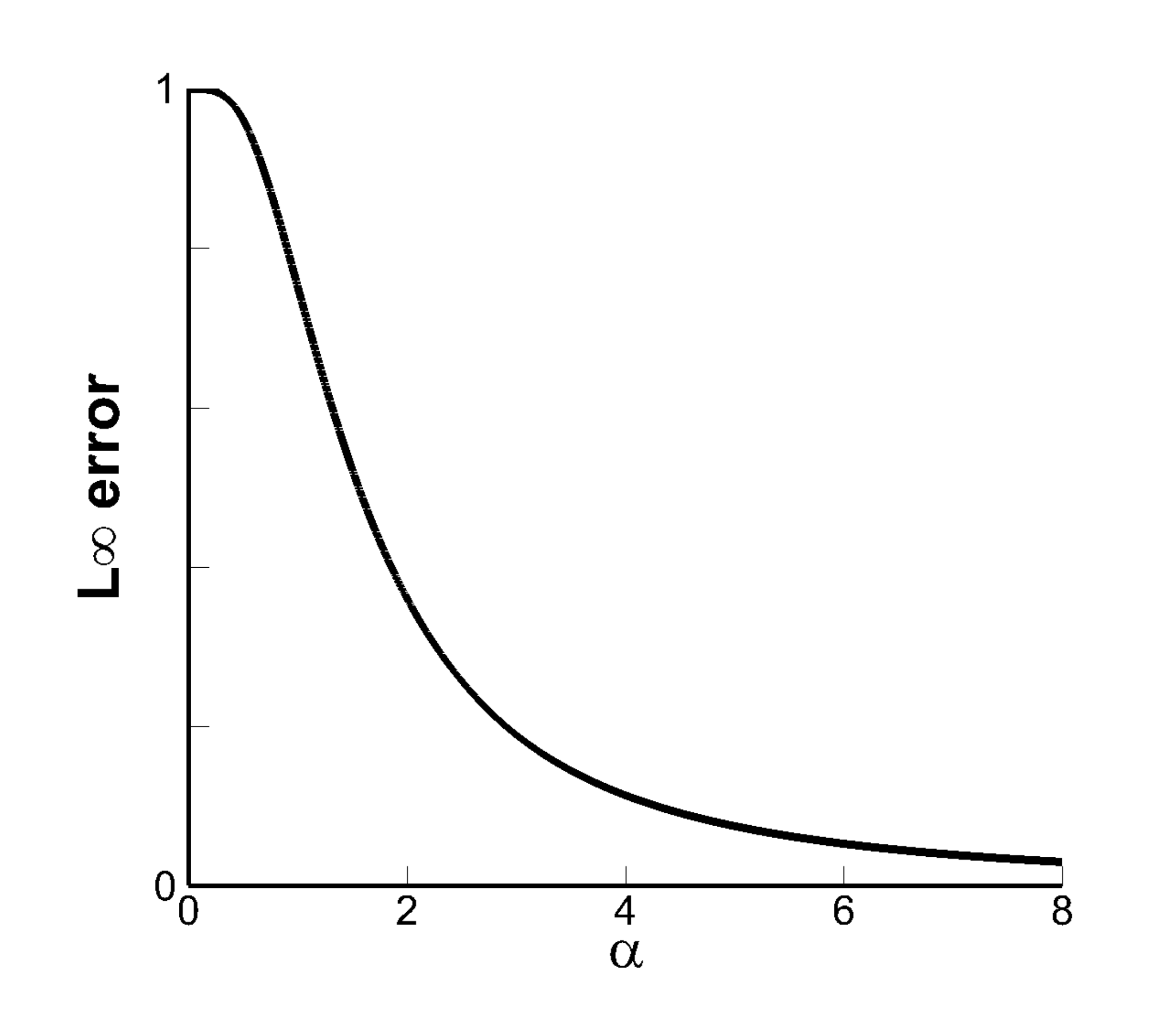}\label{fig:err3_k1}}
	\subfigure[$k=2$]{\includegraphics[width=.3\textwidth]{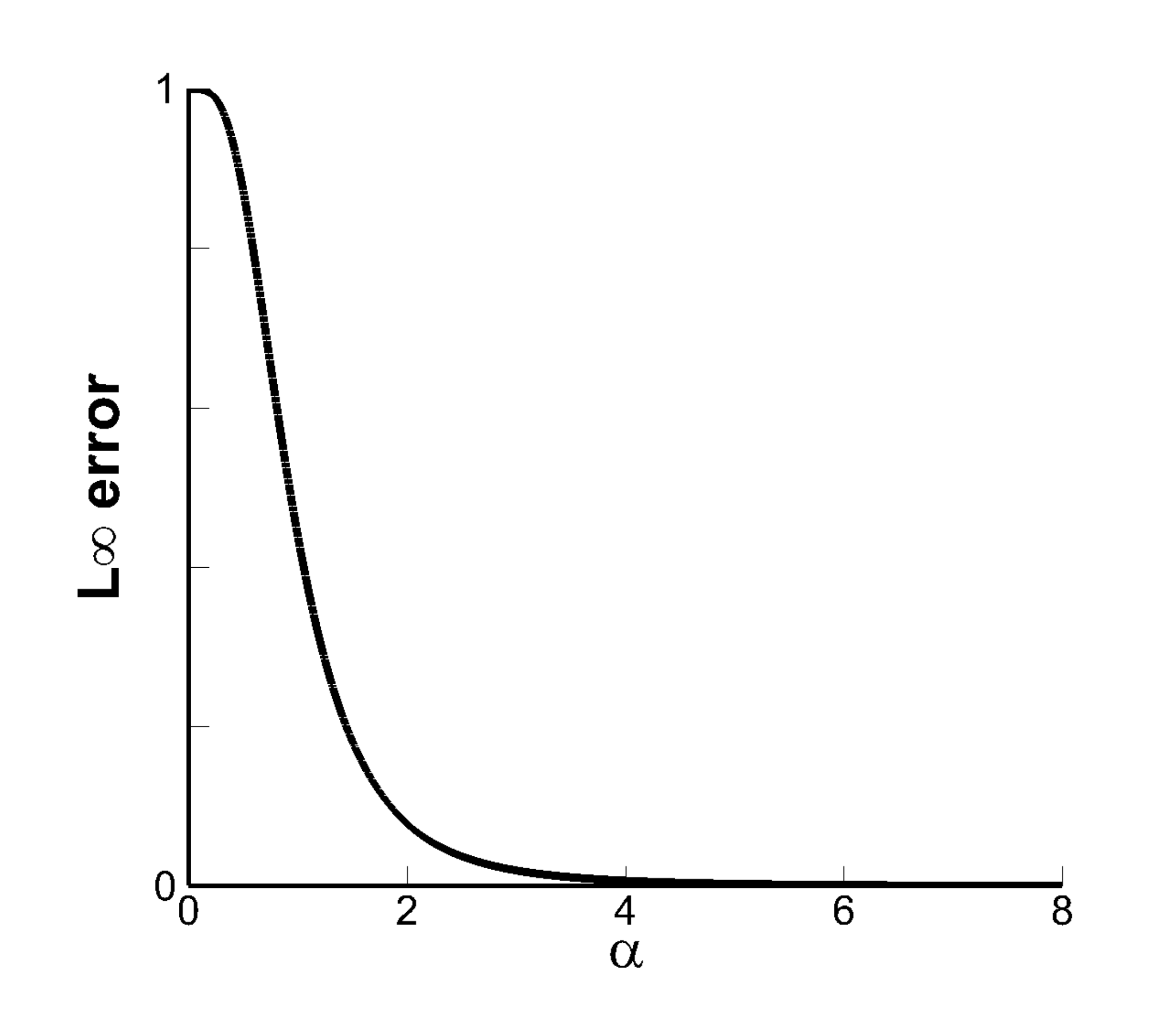}}
	\subfigure[$k=3$]{\includegraphics[width=.3\textwidth]{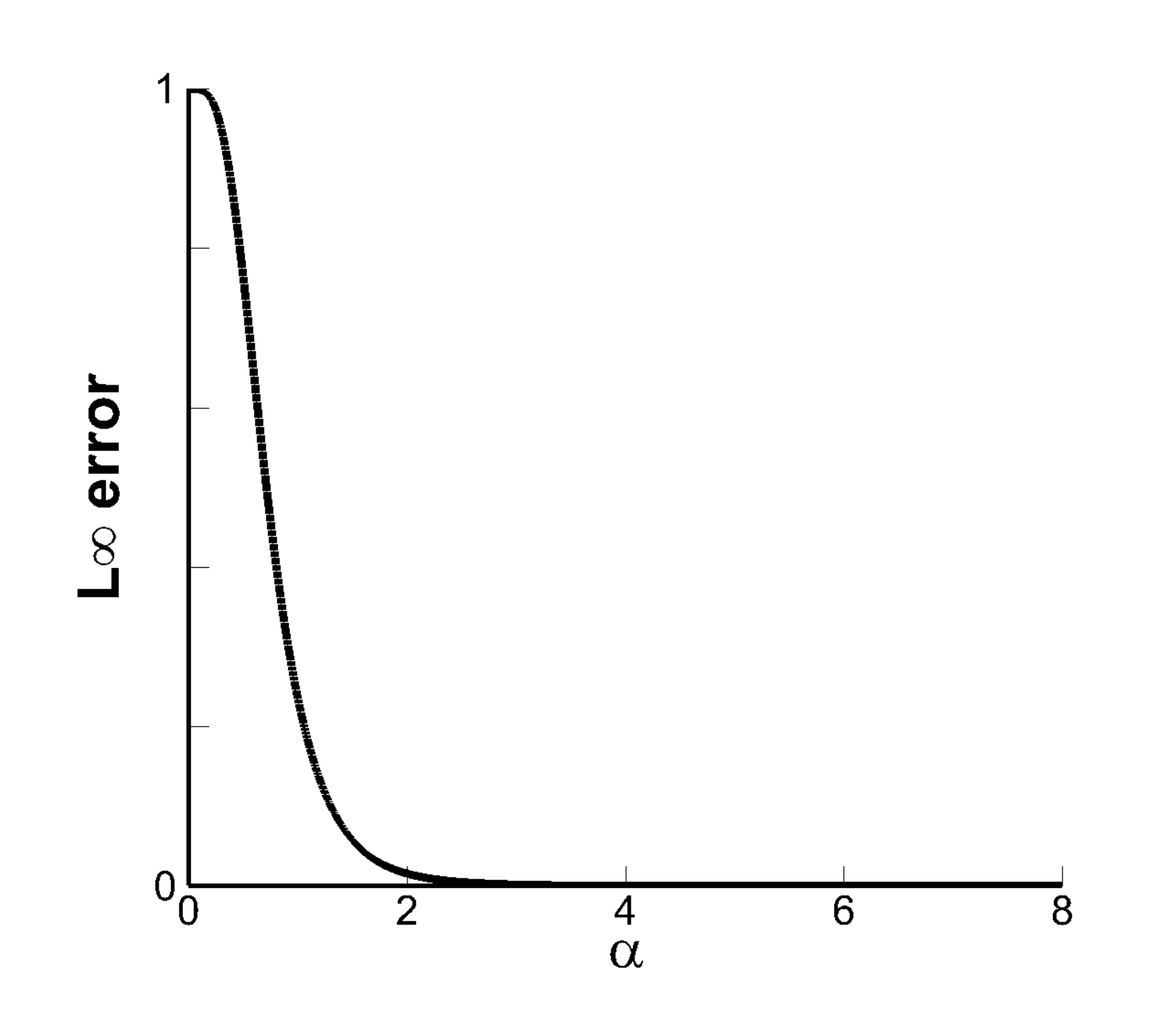}}
	\caption{$\Vert (\sin x)_{xx}-\mathcal{H}^k_3[\sin x,1,\alpha] \Vert_\infty$, with $\mathcal{H}^k_3$ given in \eqref{eq:1D}.}
\label{figerr3}
\end{figure}

On the other hand, we want to remark that computational complexity of $\mathcal{H}^k_3$ \eqref{eq:1D} is the same as $\mH^k_1$ \eqref{eq:newmth1}, and only half of that of $\mathcal{H}^k_2$ \eqref{eq:method2}, when they have the same order of accuracy. This is another advantage of this novel scheme. 

Therefore, in this work, we will employ $\mathcal{H}^k_3$ to approach the diffusion term. Moreover, we choose the parameter 
\begin{align}
	\alpha=\sqrt{ \frac{\beta}{ c\Delta t}}, \qquad \text{with} \quad c=\max_{u,x} |A(u,x,t)|.
\end{align}
Therefore, 
\begin{align}
\label{eq:H3}
	\mathcal{H}^k_3[u,A](x) 
	=& \frac{\beta}{4c\Delta t} \sum\limits_{p=1}^k\mathcal{D}_0^{p-1} (\mathcal{D}_L-\mathcal{D}_R) \left[ A   \sum\limits_{p=1}^k\mathcal{D}_0^{p-1} (\mathcal{D}_L-\mathcal{D}_R) [ u,\sqrt{\frac{\beta}{c\Delta t}} ], \sqrt{\frac{\beta}{c\Delta t}} \right](x) \notag\\
	=& (A(u,x) u_x)_x + \mO(\Delta t^k).
\end{align}

\subsection{Stability}

In this section, we will analyze the linear stability for the 1D equation \eqref{eq:equation2}, using $\mathcal{H}^k_3$ \eqref{eq:H3} for spatial derivative and the classic explicit SSP RK methods to advance $u^n$ to $u^{n+1}$. Even though the explicit method is used for time integration, we will show that the semi-discrete schemes can be A-stable and hence allowing for large time step evolution if $\beta$
in \eqref{eq:H3} is appropriately chosen.

To achieve $k$-th order accuracy in time, we should employ the $k$-th order SSP RK method as well as the $k$-th partial sum in $\mathcal{H}^k_3$. Note that high order SSP RK method \eqref{eq:rk2} and \eqref{eq:rk3} are linear combination of the first order Euler forward \eqref{eq:rk1}. Hence, here we only establish linear stability of the schemes $\mathcal{H}^k_3$ with first order Euler forward, which is given in the following theorem.

\begin{thm}
\label{th:stb1}
Consider the linear equation $u_t=A u_{xx}$, $A>0$, with periodic boundary conditions.
When the Euler forward time discretization coupling with the partial sum $\mathcal{H}^k_3$ in \eqref{eq:H3},
the scheme can be A-stable if $\beta$ satisfies $0<\beta \leq \beta_{2,k,max}$. Here, $\beta_{2,k,max}$ is a positive constant which only depends on $k$. 
The constant $\beta_{2,k,max}$ for $k = 1, 2, 3$ are summarized in Table \ref{tab1}.
\end{thm}

\begin{proof} 
Suppose the solution is enough smooth, then
$$
u(x,t)=\hat{u}(t) e^{i\eta x}.
$$
We can obtain the amplification factor $\hat{Q}$ using Von Neumann analysis.
Plugging the above formula in the definition of $\mathcal{L}_L$ and $\mathcal{D}_L$, we have
\begin{align*}
	\mathcal{L}_L[u,\alpha] =(\mathcal{I}+\frac{1}{\alpha}\partial_x)u
	=(1+\frac{i\eta}{\alpha})u,\quad \text{and} \quad
	\mathcal{D}_L[u,\alpha]=(\mathcal{I}-\mathcal{L}^{-1}_L)[u,\alpha]
	=\frac{\frac{i\eta}{\alpha}}{1+\frac{i\eta}{\alpha}}u.
\end{align*}
For brevity, let $z=\frac{\eta}{\alpha}$, and then we have $\mathcal{D}_L[u,\alpha]=\frac{iz}{1+iz}u$.
Similarly, we can obtain that
\begin{align*}
 	\mathcal{D}_R[u,\alpha]=-\frac{iz}{1-iz}u, \qquad
 	\frac{1}{2}(\mathcal{D}_L-\mathcal{D}_R)[u,\alpha]=\frac{iz}{1+z^2}u,\qquad 
 	\mathcal{D}_0[u,\alpha] 
 	= \frac{z^2}{1+z^2}u.
\end{align*}
Note that the scheme is given as 
\begin{align*}
	u^{n+1}=u^n + \Delta t A \frac{\alpha^2}{4} \left( \sum\limits_{p=1}^{k}\mathcal{D}_{0}^{p-1} (\mathcal{D}_L-\mathcal{D}_R) \right)^2[u^n,\alpha](x).
\end{align*}
Hence, with the sum formula of infinite 
sequence, the amplification factor can be written as
\begin{align*}
	\hat{Q}
	&= 1+ A \alpha^2\Delta t \left( \sum\limits_{p=1}^{k} \left(\frac{z^2}{1+z^2}\right)^{p-1} \frac{iz}{1+z^2} \right)^2\\
	&=1+\beta \left(iz\left(1-\left(\frac{z^2}{1+z^2}\right)^k\right)\right)^2\\
	&=1-\beta z^2\left(1-\left(\frac{z^2}{1+z^2}\right)^k\right)^2.
\end{align*}
Define $S_k(z)=z^2\left(1-\left(\frac{z^2}{1+z^2}\right)^k\right)^2$, so that $\hat{Q}=1-\beta S_k(z)$. The scheme is stable when $\vert\hat{Q}\vert\leq 1$, which means $\beta S_k(z) \leq 2$ for any $z\in\mathbb{R}$.
 
It is easy to find that $S_k(z)$ is an even function with respect to $z$. So we only need to consider $z\geq0$. We divide $S_k$ into two terms, $z^2\left(1-\left(\frac{z^2}{1+z^2}\right)^k\right)$ and $\left(1-\left(\frac{z^2}{1+z^2}\right)^k\right)$. Then, study the monotonicity or upper bound of those two factors, respectively.
Note that
\begin{align*}
z^2\left(1-\left(\frac{z^2}{1+z^2}\right)^k\right)=\sum\limits_{p=1}^k\left(\frac{z^2}{1+z^2}\right)^p,
\end{align*}
and
$$
\frac{{\rm d}}{{\rm d}z}\left(\frac{z^2}{1+z^2}\right)^p=\frac{2pz}{(1+z^2)^2}\left(\frac{z^2}{1+z^2}\right)^{p-1}\geq 0, \quad \text{for} \, z\geq 0.
$$
Hence, we have that $z^2\left(1-\left(\frac{z^2}{1+z^2}\right)^k\right)$ 
is nonnegative and monotonous increasing of $z$, and the upper bound is 
$$
\lim\limits_{z\rightarrow+\infty}\sum\limits_{p=1}^k\left(\frac{z^2}{1+z^2}\right)^p=\sum\limits_{p=1}^k\lim\limits_{z\rightarrow+\infty}\left(\frac{z^2}{1+z^2}\right)^p=k.
$$
On the other hand, it is obviously that $1-\left(\frac{z^2}{1+z^2}\right)^k$ is monotone decreasing of $z$ and tending to zero. Therefore we have $0\leq S_k\leq M_k$ where $M_k$ is a positive constant which only depends on $k$. Consequently, the scheme is A-stable if $\beta\leq\frac{2}{M_k}=:\beta_{2,k,max}$.
\end{proof}

\begin{table}[htb]
	\caption{\label{tab1} $\beta_{2,k,max}$ in Theorem \ref{th:stb1} for $k=1,2,3$}
	\bigskip
	\centering
	\begin{tabular}{cccc}\hline
	$k$	&	1	&	2	&	3	\\\hline
	$\beta_{2,k,max}$	&	8	&	3.2275	&	1.9800	\\\hline
	\end{tabular}
\end{table}

\begin{rem}
	We want to remark that in \cite{christlieb2017kernel}, we found that the second order derivative $\partial_{xx}$ can be represented by the infinity series of $\mD_{0}$, and the heat equation $u_t=A u_{xx}$ can be simulated as
	\begin{align}
	\label{eq:uxx_old}
		A u_{xx} = - \left(\sqrt{\frac{\beta}{A\Delta t}}\right)^2 \sum_{p=1}^{k}  \mD_{0}^p [Au, \sqrt{\frac{\beta}{A\Delta t}}] + \mathcal{O}(\Delta t^k).
	\end{align}
	It was proved that the scheme is A-stable when employing the $k$-th order SSP RK method and the $k$-th partial sum with $\beta$ in a given interval. However, the upper bounds of $\beta$ are much smaller than those in Table \ref{tab1}. For instance, when $k=1$, $\beta_{\max}=2$, which is only quarter of that of the new scheme. Hence, the proposed scheme has smaller errors and is more efficient. The comparison will be showed in numerical simulation.
\end{rem}

\begin{rem}
	Consider the parabolic problem $u_t=(A(u,x,t) u_{x})_{x}+ B(u,x,t) u_{x}$, where $A(x,t)>0$. Then, we approximate the spatial derivatives by scheme 
	\begin{align}
	\label{eq:H}
		\mH^{k}[u](x) 
		=&  \frac{\alpha_{0}^2}{4} \sum\limits_{p=1}^k\mathcal{D}_0^{p-1} (\mathcal{D}_L-\mathcal{D}_R) \left[ A(u,x,t)  \sum\limits_{q=1}^k\mathcal{D}_0^{q-1} (\mathcal{D}_L-\mathcal{D}_R) [ u,\alpha_0 ],\alpha_0 \right](x) \notag\\
		& + \frac{1}{2}B(u,x,t) \left( u^{-}_x + u_x^{+} \right) + \frac{1}{2}r \left( u^{+}_x - u_x^{-} \right),  
	\end{align}
	with $u^{\pm}_{x}$ are given in \eqref{eq:ux_pn} and the parameters
	\begin{equation}
	\begin{aligned}
	& \alpha_{L}=\alpha_{R}= \frac{\beta_1}{r\Delta t}, \quad r=\max|B(u,x,t)|, \\
	& \alpha_{0}=\sqrt{\frac{\beta_2}{c\Delta t}}, \quad  c=\max |A(u,x,t)|.
	\end{aligned}
	\end{equation}
	Then, $\mH^{k}[u](x) =(A(u,x,t) u_{x})_{x}+ B(u,x,t) u_{x} +\mO(\Delta t^k)$. 
	Moreover, consider the linear function, where $A$ and $B$ are both constants. The scheme is A-stable if we employs the $k$-th order SSP RK method and the $k$-th partial sum with $\beta_1\leq \frac{1}{2} \beta_{1,k,\max}$ and $ \beta_2\leq\frac{1}{2}\beta_{2,k,\max}$, for $k = 1, 2, 3$. 
\end{rem}

\subsection{Space Discretization}
In the previous sections, we always consider the partial sum with exact integration. Here, we present the details about the spatial discretization of $\mathcal{H}^k[u]$ in \eqref{eq:H}.
Suppose the domain $[a,b]$ is divided by $N+1$ uniformly distributed grid points
$$a=x_0<x_1<\cdots<x_{N-1}<x_N=b,$$
with mesh size $\Delta x=\frac{b-a}{N}$. Denote $u^n_i$ as the numerical solution at spatial location $x_{i}$ at time level $t^{n}$. 
On each grid point $x_{i}$, we further denote $L^{*}[v,\alpha](x_{i})$ as $L^{*}_{i}$, where $*$ can be $0$, $L$ and $R$. Note that the convolution integrals $I^{L}_{i}$ and $I^{R}_{i}$ satisfy a recursive relation 
\begin{equation}
	\label{eq:recursive}
	\begin{aligned}
	& I^L_i = I^L_{i-1}e^{-\alpha\Delta x} + J^L_i,\quad i=1,\ldots,N, \quad I^L_0 = 0, \\
	& I^R_i = I^R_{i+1}e^{-\alpha\Delta x} + J^R_i,\quad i=0,\ldots,N-1, \quad I^R_N = 0,
	\end{aligned}
\end{equation}
respectively, where
\begin{align}
\label{eq:JLR}
J^L_{i} =  \alpha \int_{x_{i-1}}^{x_{i}} v(y)e^{-\alpha (x_{i}-y)}dy,\ \ \ \
J^R_{i} =  \alpha \int_{x_{i}}^{x_{i+1}} v(y)e^{-\alpha (y-x_{i})}dy.
\end{align}
Therefore, once we have computed $J^{L}_{i}$ and $J^{R}_{i}$ for all $i$, we then can obtain $I^{L}_{i}$ and $I^{R}_{i}$ via the recursive relation.
In addition, note that the convolution integral $I^{0}[v,\alpha](x)$ can be split into $I^{L}[v,\alpha](x)$ and $I^{R}[v,\alpha](x)$,
$$ I^0_i =\frac{1}{2}( I^L_i + I^R_i ), \quad i=0,\ldots,N.$$
Thus, $I^0_i$ can be evaluated in the same way as $I^L_i$ and $I^R_i $. 

For the parabolic equation, the solution is smooth in space. Hence, here we only need the quadrature based on the interpolation. We take the quadrature rule for $J^{L}_{i}$ with six points as an example. Suppose $p(x)$ is the unique polynomial of degree at most five which interpolates $v$ at $\{x_{i-3}, \ldots, x_{i+2}\}$. Then,
\begin{subequations}
\label{eq:quar}
\begin{align}
	\label{eq:quar1}
	J^{L}_{i} \approx \alpha \int_{x_{i-1}}^{x_{i}} p(y) e^{-\alpha(x_i-y)} dy 
	=\sum\limits_{j=0}^{5} C_{-3+j} v_{i-3+j},
\end{align}
where the coefficients $C_{-3+j}$ depend on $\alpha$ and the cell size $\Delta x$, but not on $v$. These coefficients would be given out in Appendix \ref{appen1}.
The process to obtain $J^{R}_{i}$ is mirror-symmetric to that of $J^{L}_{i}$ with respect to point $x_{i}$
\begin{align} \label{eq:quar2}
J^{R}_{i} \approx \sum\limits_{j=0}^{5} C_{-3+j} v_{i+3-j},
\end{align}
\end{subequations}  

We want to remark that when the solution has discontinuities or sharp fronts, for instance, the solution of degenerate parabolic equations,
the weighted essentially non-oscillatory (WENO) integration and a nonlinear filter can be used to control the numerical oscillation near shock and achieve high order accuracy in smooth regions. Details can be found in \cite{christlieb2017kernel, christlieb2019kernel}.

	Consider the fully discrete scheme \eqref{eq:H} with $k$-th order SSP RK scheme and the quadrature rule \eqref{eq:quar}, the linear stability property can be obtained by the Fourier analysis under the assumption that $u^n_j=\hat{u}^n e^{i\kappa x_j}$. Again, we only consider the linear diffusion equation $u_t=u_{xx}$, since the analysis for linear advection equation $u_t=u_x$ is given in \cite{christlieb2017kernel}. 
It is straightforward to check that the amplification factor $\lambda$ for the linear diffusion equation depends on $\beta$, $\kappa\Delta x$ and $\Delta t/\Delta x^2$. 
Moreover, we can verify that, if $0<\beta\leq \beta_{2,k,max}$, for $k = 1, 2, 3$, then $|\lambda|\leq1$ for any $\kappa \Delta x\in[0, 2\pi]$, $\Delta t$ and $\Delta x$, indicating the fully discrete scheme is unconditionally stable.

\section{Two-dimensional Implementation}

In this section, we will consider the two-dimensional problem
\begin{align}
\label{eq:2D}
	u_t =& (A_{11}(u,x,y,t)u_x)_x + (A_{22}(u,x,y,y)u_y)_y + (A_{12}(u,x,y,t)u_x)_y + (A_{21}(u,x,y,t)u_y)_x \notag\\
	& + B_{1}(u,x,y,t) u_x + B_{2}(u,x,y,t) u_y + C(u,x,y,t).
\end{align}
Let $(x_i, y_j)$ be the node of a 2D orthogonal grid. Here, we use uniform grid in each direction, with mesh sizes $\Delta x=x_{i}-x_{i-1}$ and $\Delta y=y_{j}-y_{j-1}$.
Each terms in \eqref{eq:2D} can be directly approximated by the proposed 1D formulation in a dimension-by-dimension fashion, namely, approximating $\partial_{x}$ for fixed $y_j$ and approximating $\partial_{y}$ for fixed $x_i$. More specifically, for the transport parts,
\begin{equation}
\begin{aligned}
	& B_1 u_x|_{(x_i,y_j)} \approx \frac{1}{2}B_1(u,x,y,t) \left( u^{-}_x + u_x^{+} \right) + \frac{1}{2}r_{x} \left( u^{+}_x - u_x^{-} \right),  \\
	& B_2 u_y|_{(x_i,y_j)}  \approx \frac{1}{2}B_2(u,x,y,t) \left( u^{-}_y + u_y^{+} \right) + \frac{1}{2}r_{y} \left( u^{+}_y - u_y^{-} \right),  
\end{aligned}
\end{equation}
where, 
\begin{align*}
	&	u^{-}_x|_{(x_i,y_j)} = \alpha_{L,x} \sum_{p=1}^{k} \mathcal{D}_{L,x}^{p} [ u(\cdot, y_j), \alpha_{L,x}] (x_i), \quad 
	u^{+}_x|_{(x_i,y_j)} = -\alpha_{R,x} \sum_{p=1}^{k} \mathcal{D}_{R,x}^{p} [ u(\cdot,y_j),\alpha_{R,x}] (x_i), \\
	& u^{-}_y|_{(x_i,y_j)} = \alpha_{L,y} \sum_{p=1}^{k} \mathcal{D}_{L,y}^{p} [ u(x_i, \cdot), \alpha_{L,y}] (y_j), \quad 
	u^{+}_y|_{(x_i,y_j)} = -\alpha_{R,y} \sum_{p=1}^{k} \mathcal{D}_{R,y}^{p} [ u(x_i, \cdot),\alpha_{R,y}] (y_j),
\end{align*}
or with a modified term for $k = 3$, and parameters 
\begin{align*}
	\alpha_{L,x}=\alpha_{R,x}=\beta_1/(r_{x}\Delta t), \quad r_{x}=\max_{u,x,y} |B_1(u,x,y,t)| \\
	\alpha_{L,y}=\alpha_{R,y}=\beta_1/(r_{y}\Delta t), \quad r_{y}=\max_{u,x,y} |B_2(u,x,y,t)|.
\end{align*}
And for the diffusion terms,
\begin{equation}
\label{eq:uxx}
\begin{aligned}
	& (A_{11}u_x)_x\vert_{(x_i,y_j)} \approx \frac{\alpha_{0,x}}{2} \sum\limits_{p=1}^k \mathcal{D}_{0,x}^{p-1} (\mathcal{D}_{L,x}-\mathcal{D}_{R,x}) [A_{11}(u(\cdot,y_j,t),\cdot,y_j,t) \, w_1(\cdot, y_j,t), \alpha_{0,x}] (x_i),  \\
	& (A_{22} u_y)_y\vert_{(x_i,y_j)} \approx \frac{\alpha_{0,y}}{2} \sum\limits_{p=1}^k \mathcal{D}_{0,y}^{p-1} (\mathcal{D}_{L,y}-\mathcal{D}_{R,y}) [A_{22}(u(x_i,\cdot,t),x_i,\cdot,t) \, w_2(x_i,\cdot,t),\alpha_{0,y}](y_i), \\
	& (A_{12}u_x)_y\vert_{(x_i,y_j)} \approx \frac{\alpha_{0,y}}{2} \sum\limits_{p=1}^k \mathcal{D}_{0,y}^{p-1} (\mathcal{D}_{L,y}-\mathcal{D}_{R,y}) [A_{12}(u(x_i, \cdot,t),x_i, \cdot, t) \, w_1( x_i, \cdot, t), \alpha_{0,y}] (y_j),  \\
	& (A_{21}u_y)_x\vert_{(x_i,y_j)} \approx \frac{\alpha_{0,x}}{2} \sum\limits_{p=1}^k \mathcal{D}_{0,x}^{p-1} (\mathcal{D}_{L,x}-\mathcal{D}_{R,x}) [A_{21}(u(\cdot,y_j,t),\cdot,y_j,t) \, w_2(\cdot, y_j,t), \alpha_{0,x}] (x_i) , \\
	& w_1\vert_{(x_i,y_j)} = u_x\vert_{(x_i,y_j)} \approx \frac{\alpha_{0,x}}{2} \sum\limits_{p=1}^k \mathcal{D}_{0,x}^{p-1} (\mathcal{D}_{L,x}-\mathcal{D}_{R,x}) [ u(\cdot,y_j,t), \alpha_{0,x}](x_i)  \\
	& w_2\vert_{(x_i,y_j)} = u_y\vert_{(x_i,y_j)}  \approx \frac{\alpha_{0,y}}{2} \sum\limits_{q=1}^k \mathcal{D}_{0,y}^{q-1} (\mathcal{D}_{L,y}-\mathcal{D}_{R,y}) [u(x_i,\cdot,t),\alpha_{0,y}] (y_i).
\end{aligned}
\end{equation}
where 
\begin{align*}
	& \alpha_{0,x}=\sqrt{\beta_2/(c_x\Delta t)}, \quad c_x=\max\limits_{u,x,y}\vert A_{11}(u,x,y)\vert,  \\
	& \alpha_{0,y}=\sqrt{\beta_2/(c_y\Delta t)}, \quad c_y=\max\limits_{u,x,y}\vert A_{22}(u,x,y)\vert.
\end{align*}

%

Similarly, in the two-dimensional case, we can choose the parameters $\beta_1$ and $\beta_2$ carefully such that the scheme is A-stable. In particular, considering the diffusion terms only, we have the following theorem.

\begin{thm}
\label{th:erruxy}
	Consider the linear parabolic equation with periodic boundary
	\begin{align}
	u_t=A_{11} u_{xx}+ (A_{12}+A_{21}) u_{xy} + A_{22} u_{yy},  \quad (x,y)\in [0,2\pi]^2,
	\end{align}
	where the coefficients $A_{ij}$ are constants. Suppose the scheme is constructed by the partial sums \eqref{eq:uxx} with $k$ terms and combined with the Euler forward. 
	\begin{enumerate}
		\item If $A_{12}=A_{21}=0$, the scheme is A-stable when we take $0< \beta \leq \beta_{k,\max} = \frac{1}{2}\beta_{2,k\max}$.
		
		\item Otherwise, the scheme is A-stable if we take $0< \beta \leq \beta_{k,\max} = \frac{1}{4}\beta_{2,k\max}$.
	\end{enumerate} 
\end{thm}

\begin{proof}
	Here, we only give the proof of the second case, which is more general. Suppose $u$ is smooth enough that can be written as $u(x,y,t)=\hat{u}(t)e^{i\xi x+i\eta y}$. Similar to the proof of Theorem \ref{th:stb1}, we can use the Von Neumann analysis and obtain the amplification factor $\hat{Q}$ 
	\begin{align*}
	\hat{Q} 
	=& 1-A_{11} \alpha_{0,x}^2\Delta t \left(\sum\limits_{p=1}^{k}\frac{z_1^{2p-1}}{(1+z_1^2)^p}\right)^2
	-A_{22} \alpha_{0,y}^2 \Delta t \left(\sum\limits_{q=1}^{k}\frac{z_2^{2q-1}}{(1+z_2^2)^q}\right)^2  \\
	& - (A_{12}+A_{21}) \alpha_{0,x} \alpha_{0,y} \Delta t \left(\sum\limits_{p=1}^{k}\frac{z_1^{2p-1}}{(1+z_1^2)^p}\right)\left(\sum\limits_{q=1}^{k}\frac{z_2^{2q-1}}{(1+z_2^2)^q}\right),
	\end{align*}
	where $z_1=\xi/\alpha_{0,x}$ and $z_2=\eta/\alpha_{0,y}$.
	Let $R_x=\sum\limits_{p=1}^{k}\frac{z_1^{2p-1}}{(1+z_1^2)^p}$ and $R_y=\sum\limits_{q=1}^{k}\frac{z_2^{2q-1}}{(1+z_2^2)^q}$. 
	Note that the function is parabolic if there exists a constant $\theta>0$ such that
	$$ \begin{pmatrix}\xi_{1} & \xi_{2}\end{pmatrix} \begin{pmatrix} A_{11} & A_{12} \\ A_{21} & A_{22} \\\end{pmatrix}  \begin{pmatrix}\xi_{1} \\ \xi_{2}\end{pmatrix} \geq \theta (\xi_{1}^2 + \xi_2^2), \quad \forall (\xi_1, \xi_2). $$
	This means $A_{11}>0$, $A_{22}>0$ and $\frac{ |A_{12}+A_{21}| }{\sqrt{A_{11} A_{22}}} \leq 2$.
	Then, we have
	\begin{align*}
	\hat{Q}=& 1-A_{11}\alpha_{0,x}^2\Delta t R_x^2 -A_{22} \alpha_{0,y}^2\Delta t R_y^2 - (A_{12}+A_{21}) \alpha_{0,x} \alpha_{0,y} \Delta t R_x R_y\\
	=& 1 - \beta R_{x}^2 -\beta R_{y}^2 - \frac{A_{12}+A_{21}}{\sqrt{A_{11} A_{22}}} \beta R_x R_y
	\end{align*}

	\noindent
	Note that $$\hat{Q}\leq 	1-\beta\min\left((R_x-R_y)^2,(R_x+R_y)^2\right)\leq 1.$$
	Hence, the scheme is A-stable if $ \hat{Q}\geq-1$.
%
	Since $R_x^2\leq M_k$ and $R_{y}^2\leq M_k$. We take $\beta\leq \frac{1}{4} \beta_{2,k,\max}$, then 
	\begin{align*}
		\hat{Q} \geq& 1-\beta\max\left((R_x-R_y)^2,(R_x+R_y)^2\right)
		 \geq 1- \beta (4 M_k) \\
		 \geq& 1-\frac{1}{4} \beta_{2,k,\max} 4 M_{k} = 1-\beta_{2,k,\max} M_{k} = -1.
	\end{align*}
	Namely, the scheme is A-stable.

\end{proof}

\begin{rem}
	We want to remark that, the methods \eqref{eq:newmth1} and \eqref{eq:method2} can also be used to solve the 2D problem based on a dimension-by-dimension approach. However, we cannot prove the A-stable property of either method when $A_{12}\neq0$ or $A_{21}\neq0$. The provable A-stable property is one main advantage of this proposed scheme.
\end{rem}

\begin{rem}
	Consider the function \eqref{eq:2D} with both diffusion terms and transport terms. Suppose the scheme employs the $k$-th order SSP RK method and the $k$-th partial sum. Then, the scheme is unconditional stability if $0< \beta_1 \leq  \frac{1}{4} \beta_{1,k,\max}$ and $0< \beta_2\leq \frac{1}{8} \beta_{2,k,\max}$.
\end{rem}
\section{Numerical Results}
In this section, we show the results of our numerical experiments for the schemes to demonstrate their efficiency and efficacy. 
We take the time step as
\begin{align*}
\Delta t=\text{CFL} \cdot \Delta x,
\end{align*}
for one-dimensional problems, and 
\begin{align*}
\Delta t= \text{CFL} \cdot  \min(\Delta x, \Delta y).
\end{align*}
for two-dimensional problems.
Note that time step $\Delta t$ is chosen in a form similar to a standard MOL type method. It will enable us to conveniently test accuracy and compare the scheme with other methods. We remark that the CFL number can be chosen arbitrarily large due to the unconditional stability. 
Becasue all solutions are smooth here, the six points quadrature formula \eqref{eq:quar} without WENO is used to compute $J^{L}_{i}$ and $J^{R}_{i}$. And we always choose $\beta=\beta_{k,\max}$ for each scheme in numerical simulations.

\begin{exa}\label{ex1}
Firstly, we consider the one-dimensional heat equation 
\begin{align}
	\left\{ \begin{array}{ll}
	u_t=u_{xx}, \quad 0\leq x\leq 2\pi, \\
	u(x,0)=\sin x,
	\end{array} \right. 
\end{align}
with the $2\pi$-periodic boundary condition. This problem has the exact solution is $u(x,t)=e^{-t}\sin x$.
Here, we want to compare the efficiency of the new proposed scheme (denoted as ``new") and the original scheme \eqref{eq:uxx_old} in \cite{christlieb2017kernel} (denoted as ``old"), that is
\begin{align}
\label{equxxold}
u_{xx} \approx -\alpha^2 \sum_{p=0}^{k} \mD_{0}^{p}[u,\alpha], \quad \text{with} \quad \alpha=\sqrt{\beta_{old,k}/(c\Delta t)}.
\end{align}
$\beta_{k,max}$ for the ``new'' scheme are taken from Table \ref{tab1}, while those of the ``old'' scheme are given in Table \ref{tabold}.
In Figure \ref{fig1D}, we plot the CPU cost versus $L_{\infty}$ errors at time $T=1$, and provide such a comparison for $k=1, 2, 3$.  $CFL=1$ is used for all schemes.
It is obvious that to achieve the same error, the new scheme always cost less CPU time, which indicates the efficient of our new method. 
This is caused by the larger $\beta$ used in the proposed scheme.

\begin{table}[htb]
	\caption{\label{tabold} $\beta_{old,k,max}$ in \eqref{equxxold} for $k=1,2,3$}
	\bigskip
	\centering
	\begin{tabular}{|c|c|c|c|}\hline
	$k$	&	1	&	2	&	3	\\\hline
	$\beta_{old,k,max}$	&	2	&	1	&	0.8375	\\\hline
	\end{tabular}
\end{table}

\begin{figure}[htbp]
\centering
	\subfigure[$k=1$]{\includegraphics[width=.32\textwidth]{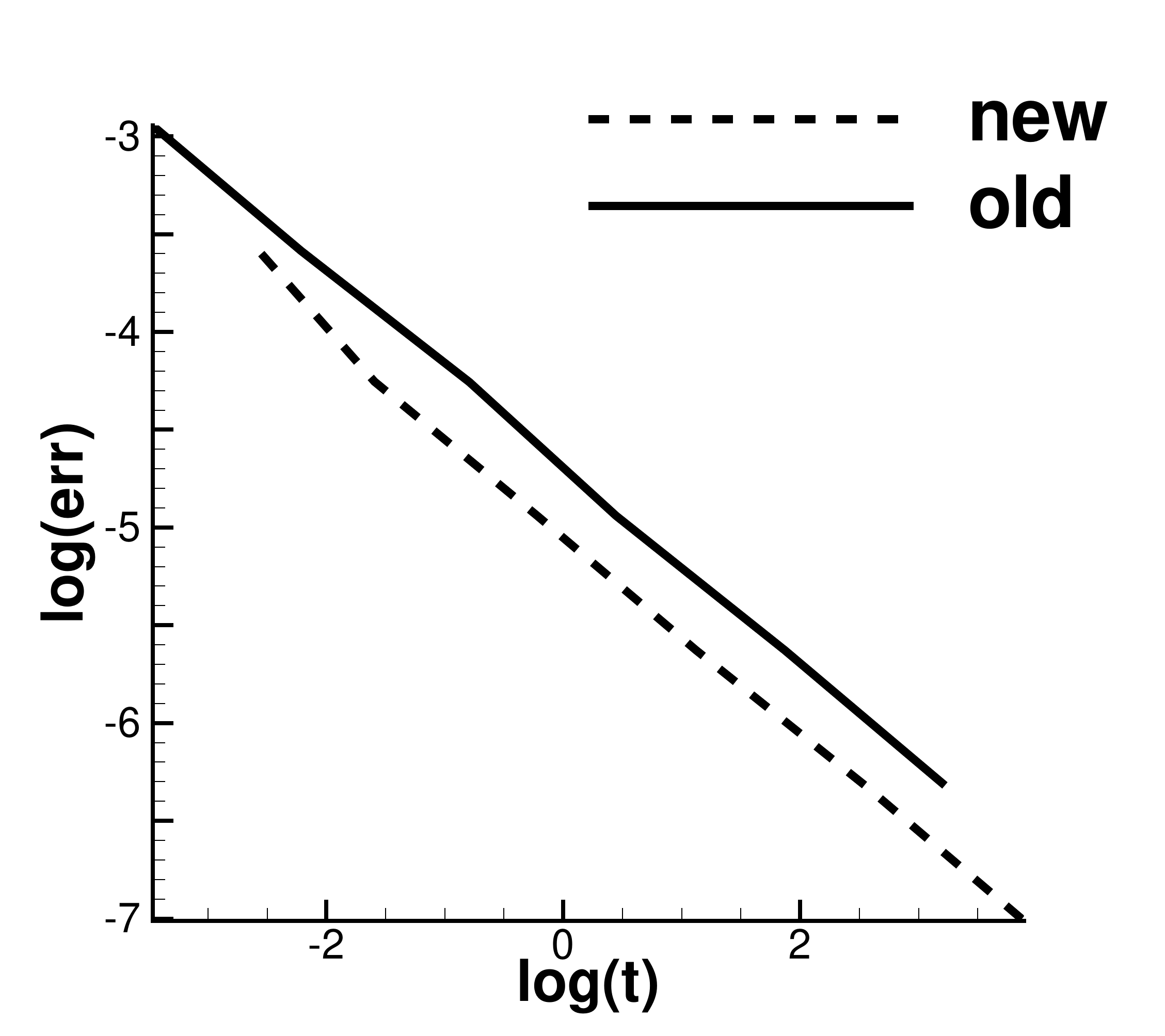}}
	\subfigure[$k=2$]{\includegraphics[width=.32\textwidth]{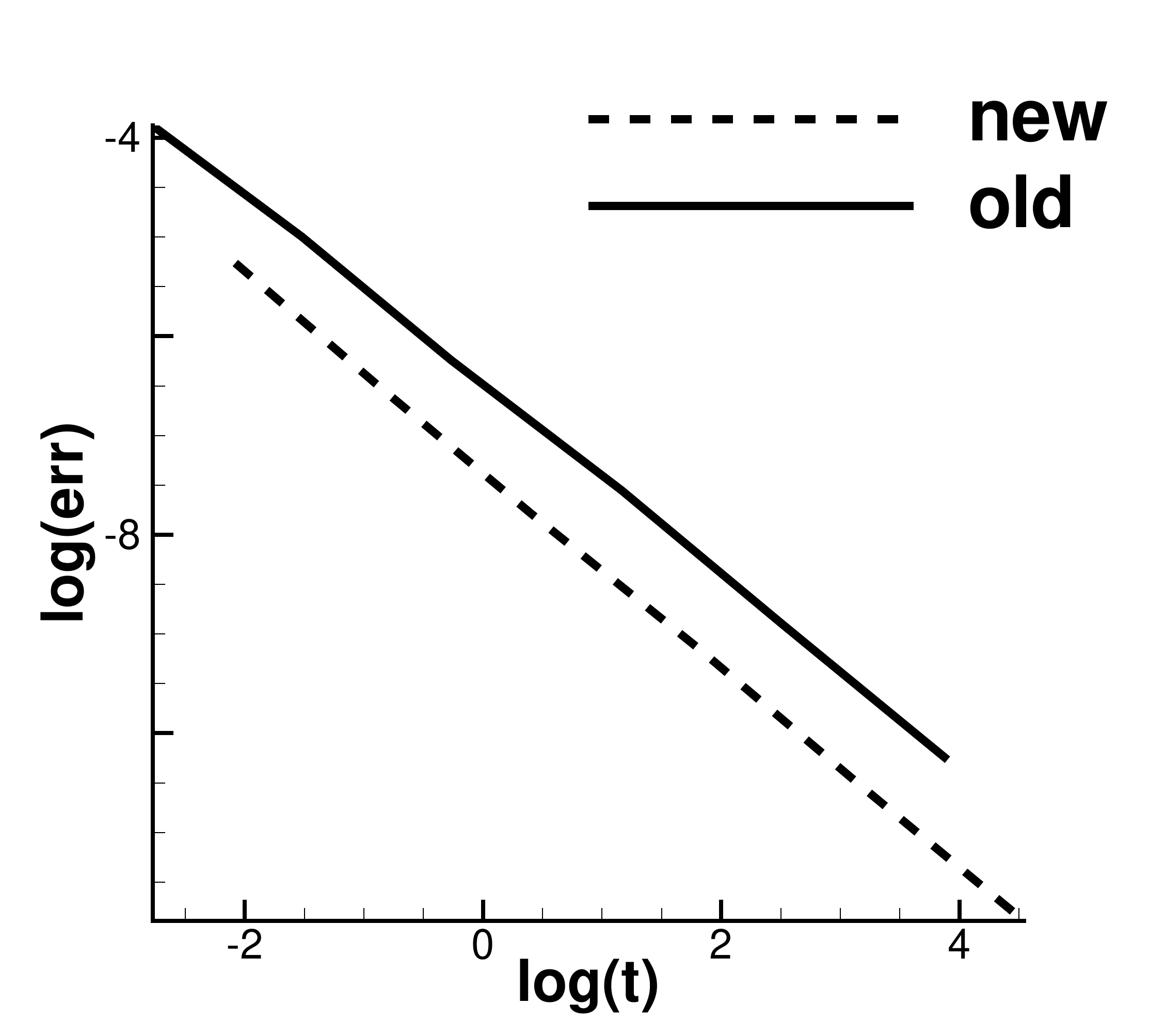}}
	\subfigure[$k=3$]{\includegraphics[width=.32\textwidth]{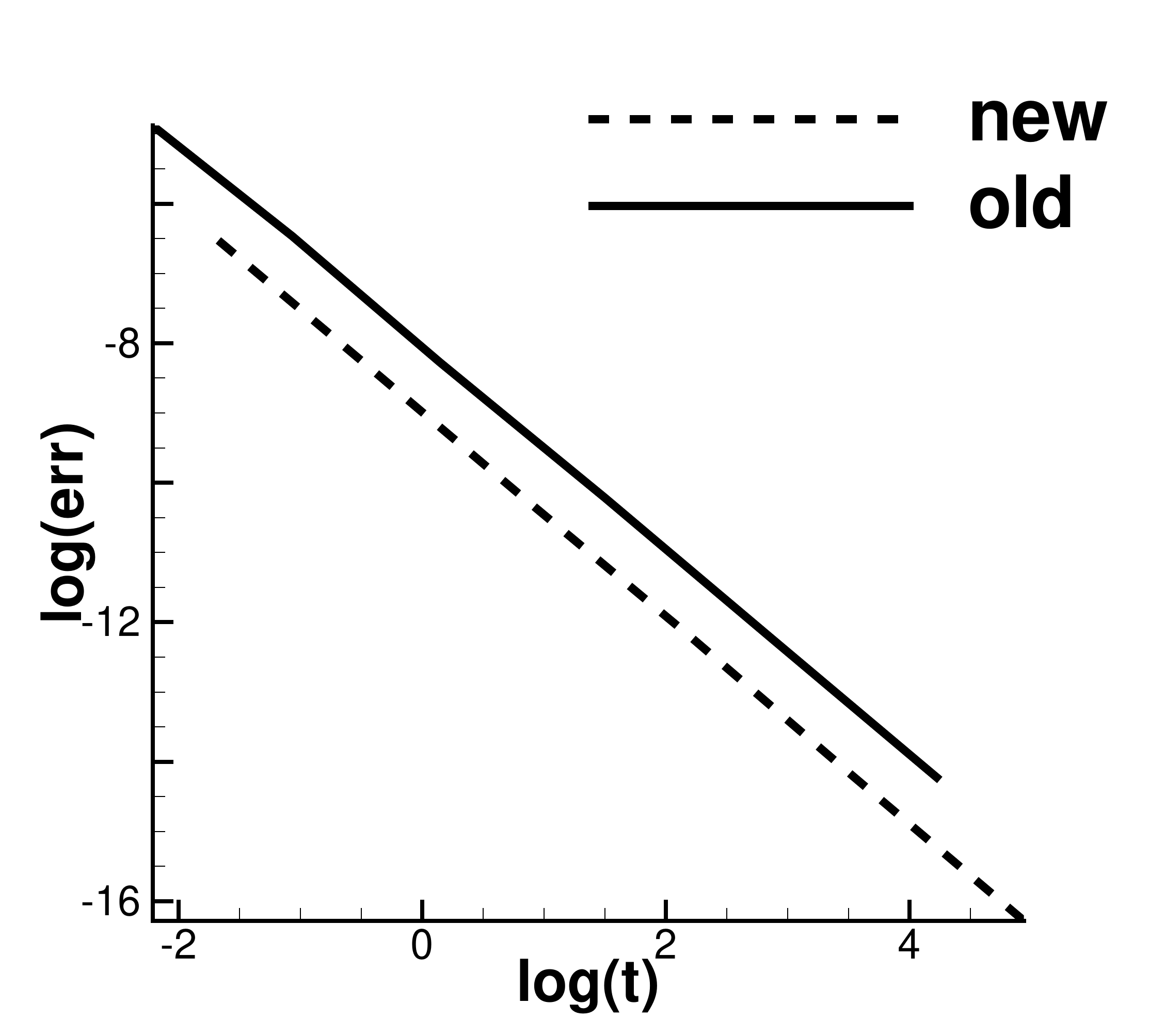}}
	\caption{Example \ref{ex1}: Comparison of CPU time against $L_\infty$ error for one-dimensional heat equation between our scheme and the scheme in \cite{christlieb2017kernel}.}
	\label{fig1D}
\end{figure}

\end{exa}

\begin{exa}\label{ex2}
Then we consider the parabolic equation with the $2\pi$-periodic boundary condition,
\begin{align}
\left\{\begin{array}{l}
u_t=(uu_x)_x-u_x+u-1-0.25\cos(2x-2t),\quad 0\leq x\leq 2\pi\\
u(x,0)=1+0.5\sin x.\\
\end{array}\right.
\end{align}
And the exact solution is $u(x,t)=1+0.5\sin(x-t)$. In Table \ref{tabexm2}, we list the errors of schemes and the associated orders of accuracy at $T=1$, with $k=1, 2, 3$. Three CFLs including 0.5, 1 and 2 are used to demonstrate the performance. It is observed that the scheme can achieve the designed order. In particular, the scheme allows for large CFL numbers due to its unconditionally stability.

\begin{table}[htbp]
\caption{Example 2: $L_\infty$ errors and orders of accuracy at $T=1$}
\centering
\bigskip
\label{tabexm2}
\begin{tabular}{|c|c|cc|cc|cc|}
\hline
\multirow{2}*{CFL}&\multirow{2}*{$N_x$}&\multicolumn{2}{c|}{$k=1$}&\multicolumn{2}{c|}{$k=2$}&\multicolumn{2}{c|}{$k=3$}\\
\cline{3-8}
~&~&error&order&error&order&error&order\\\hline
\multirow{6}*{0.5}&20 &
  0.76E-01&      -&  0.40E-01&     -&  0.77E-02&      -\\\cline{2-8}
~&                    40 &
  0.38E-01&      0.99&  0.13E-01&   1.63&  0.21E-02&      1.90\\\cline{2-8}
~&                    80 &
  0.19E-01&      1.00&  0.37E-02&      1.80&  0.43E-03&      2.29\\\cline{2-8}
~&                   160 &
  0.95E-02&      1.00&  0.10E-02&      1.84&  0.72E-04&      2.56\\\cline{2-8}
~&                   320 &
  0.47E-02&      1.00&  0.27E-03&      1.93&  0.10E-04&      2.79\\\cline{2-8}
~&                   640 &
  0.24E-02&      1.00&  0.70E-04&      1.96&  0.14E-05&      2.89\\\hline
\multirow{6}*{1}&20&
0.14E+00&-&  0.95E-01&-&  0.30E-01& -\\\cline{2-8}
~&40&
0.76E-01&      	0.88&  0.40E-01&      1.25&  0.77E-02&      1.98\\\cline{2-8}
~&80&
0.38E-01&      0.98&  0.13E-01&      1.63&  0.21E-02&      1.89\\\cline{2-8}
~&160&
0.19E-01&      1.01&  0.37E-02&      1.80&  0.43E-03&      2.29\\\cline{2-8}
~&320&
0.95E-02&      1.00&  0.10E-02&      1.84&  0.72E-04&      2.56\\\cline{2-8}
~&640&
0.47E-02&      1.00&  0.27E-03&      1.93&  0.10E-04&      2.79\\\hline
\multirow{5}*{2}&40&
 0.14E+00&      - &  0.97E-01&      - &  0.30E-01&     - \\\cline{2-8}
~&80&
0.76E-01&      0.88&  0.40E-01&      1.28&  0.78E-02&      1.97\\\cline{2-8}
~&160&
0.38E-01&      0.98&  0.13E-01&      1.63&  0.21E-02&      1.89\\\cline{2-8}
~&320&
 0.19E-01&      1.01&  0.37E-02&      1.80&  0.43E-03&      2.30\\\cline{2-8}
~&640&
0.95E-02&      1.00&  0.10E-02&      1.84&  0.72E-04&      2.56\\\hline
\end{tabular}
\end{table}

\end{exa}

\begin{exa}\label{ex4}
We use the following 2D nonlinear parabolic equation on $(x,y)\in[0,2\pi]^2$
\begin{align*}
	\left\{ \begin{array}{ll}
	u_t=(uu_x)_x+(uu_y)_y-uu_x-uu_y-u^2+f(x,y,t),  \\
	u(x,y,0)=1+0.5\sin(x+y), \\
	\end{array} \right.
\end{align*}
with $$f(x,y,t)=1.125-0.625\cos(2x+2y-2t)+0.25\sin(2x+2y-2t)+0.5\cos(x+y-t)+2\sin(x+y-t).$$
The $2\pi$-periodic boundary is considered in each direction. It is easy to verify that $u(x,y,t)=1+0.5\sin(x+y-t)$ is the exact solution. We show the $L_{\infty}$ errors and the orders of accuracy with $k=3$ in Table \ref{tabexm4}. Again, the scheme can achieve the designed order of accuracy, even with pretty large time step.
\end{exa}

\begin{table}[htb]
\caption{Example \ref{ex4}: $L_\infty$ errors and orders of $k=3$ at $T=1$}
\centering
\bigskip
\label{tabexm4}
\begin{tabular}{|c|cc|cc|cc|}
\hline
\multirow{2}*{$N_x \& N_y$} & \multicolumn{2}{c|}{$CFL=0.5$} & \multicolumn{2}{c|}{$CFL=1$} & \multicolumn{2}{c|}{$CFL=2$}\\
\cline{2-7}
~&error&order&error&order&error&order\\\hline
80 &
0.59E-02&      -&0.47E-01&      -&0.11E+00&      -\\\hline
160 &
0.12E-03&      5.64&0.80E-03&      5.87&0.46E-01&      1.24\\\hline
320 &
0.17E-04&      2.83&0.12E-03&      2.77&0.80E-03&      5.84\\\hline
640 &
0.22E-05&      2.88&0.17E-04&      2.83&0.12E-03&      2.77\\\hline
1280 &
0.30E-06&      2.92&0.22E-05&      2.88&0.17E-04&      2.83\\\hline
\end{tabular}
\end{table}

\vspace{0.5cm}
\begin{exa}\label{ex5}
	\textbf{(Schnakenberg model)}
In this example, we want to show that the scheme can also be used to solve system. The Schnakenberg system \cite{Andrew2015high} has been used to model the spatial distribution of a morphogen, which has the following form
\begin{align}
\left\{\begin{array}{l}
\displaystyle\frac{\partial C_a}{\partial t}=D_1\nabla^2 C_a+\kappa(a-C_a+C_a^2 C_i),\\
\displaystyle\frac{\partial C_i}{\partial t}=D_2\nabla^2 C_i+\kappa(b-C_a^2 C_i).
\end{array}\right.
\end{align}
Here, $C_a$ and $C_i$ represent the concentrations of activator and inhibitor, with $D_1$ and $D_2$ as the diffusion coefficients
respectively. $\kappa$, $a$ and $b$ are rate constants of biochemical reactions. Following the setup in \cite{Andrew2015high}, we take the initial conditions as
\begin{align*}
C_a(x,y,0)&=a+b+10^{-3}e^{-100((x-\frac{1}{3})^2+(y-\frac{1}{2})^2)},\\
C_i(x,y,0)&=\frac{b}{(a+b)^2}.
\end{align*}
And the parameters are 
$$\kappa=100,\quad  a = 0.1305, \quad b = 0.7695, \quad D_1 = 0.05, \quad \text{and} \quad D_2 = 1.$$
We test the problem with $k=3$ and $300\times 300$ grid points. $CFL=1$ is taken here.
$C_a$ at different times are showed in Figure \ref{figsys}, which have the similar patterns as those in \cite{Andrew2015high}.

\begin{figure}[htb]
\centering
	\subfigure[$T=0.5$]{\includegraphics[width=.32\textwidth]{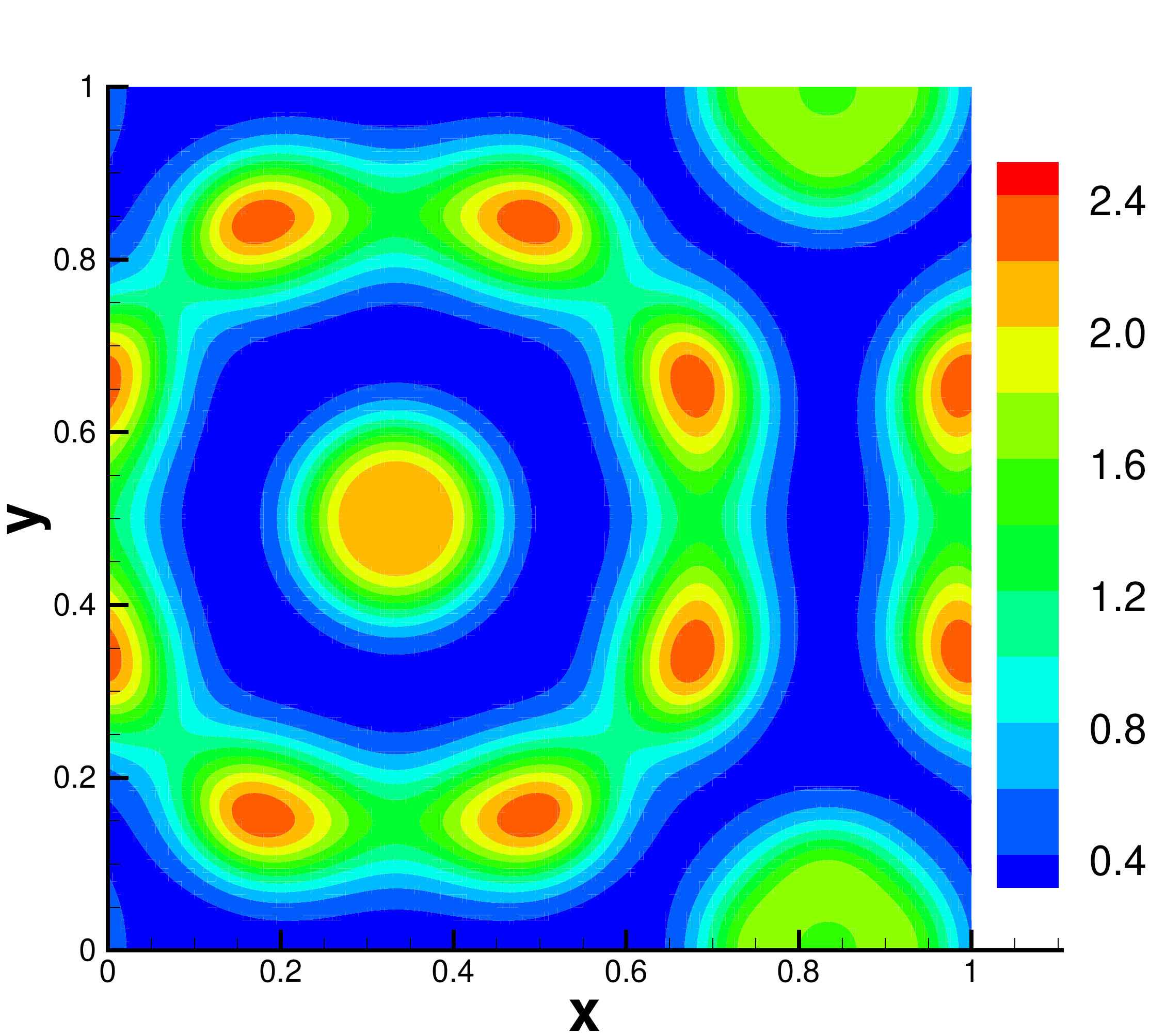}}
	\subfigure[$T=1$]{\includegraphics[width=.32\textwidth]{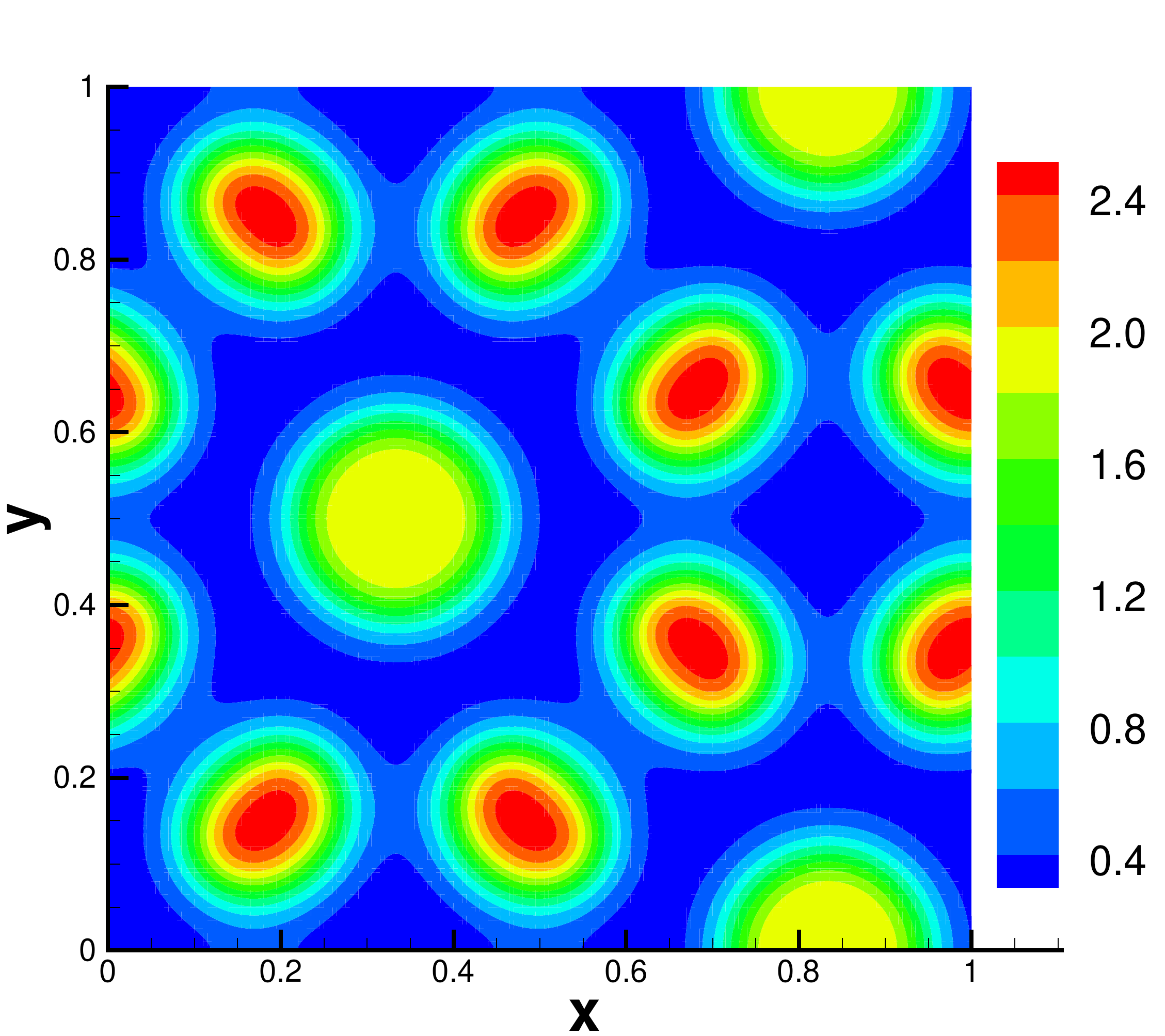}}
	\subfigure[$T=1.5$]{\includegraphics[width=.32\textwidth]{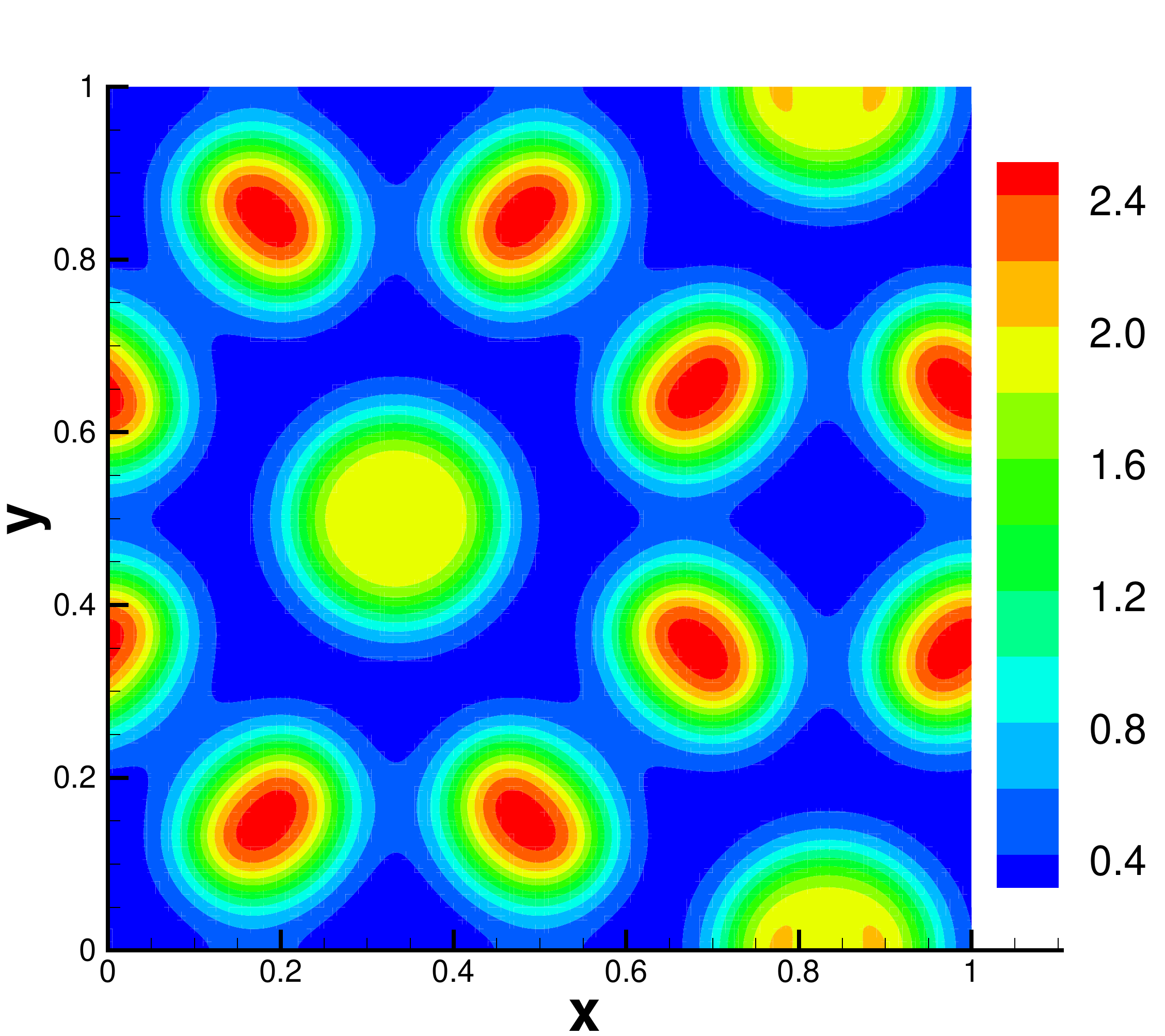}}
	\caption{Example \ref{ex5}: $C_a$ figure at different time.}
	\label{figsys}
\end{figure}
\end{exa}

\begin{exa} \label{ex6}
Finally, we show the result for equation with cross derivative terms
\begin{align*}
  \left\{\begin{array}{ll} 
  u_t=u_{xx}+u_{yy}+u_{xy}, \quad (x,y) \in(0,2\pi)^2, \\
  u(x,y,0)=\sin (x+y), \\
  \end{array} \right.
\end{align*}
and exact solution is $u(x,y,t)=e^{-3t}\sin (x+y)$. $L_{\infty}$ errors and orders of accuracy with $k=3$ are showed in Table \ref{tabexm6}, indicating the high order of accuracy and the unconditionally stable property of our scheme.

\begin{table}[htb]
\caption{Example \ref{ex6}: $L_\infty$ errors and orders at $T=1$, with $k=3$ and $\beta=0.49$.}
\centering
\bigskip
\label{tabexm6}
\begin{tabular}{|c|cc|cc|cc|}
\hline
\multirow{2}*{$N_x\&N_y$}&\multicolumn{2}{c|}{$CFL=0.5$}&\multicolumn{2}{c|}{$CFL=1$}&\multicolumn{2}{c|}{$CFL=2$}\\
\cline{2-7}
~&error&order&error&order&error&order\\\hline
40 &
0.67E-03&      -&0.29E-02&      -&0.73E-02&      -\\\hline
80 &
0.11E-03&      2.65&0.67E-03&      2.11&0.29E-02&      1.34\\\hline
160 &
0.16E-04&      2.76&0.11E-03&      2.65&0.67E-03&      2.11\\\hline
320 &
0.21E-05&      2.91&0.16E-04&      2.76&0.11E-03&      2.65\\\hline
640 &
0.27E-06&      2.95&0.21E-05&      2.91&0.16E-04&      2.76\\\hline
1280 &
0.34E-07&      2.98&0.27E-06&      2.95&0.21E-05&      2.91\\\hline
\end{tabular}
\end{table}
\end{exa}

%
\section{Conclusion}

In this paper, we proposed a novel numerical scheme to solve the nonlinear parabolic equations with variable coefficients. 
The development of the schemes was based on our previous work \cite{christlieb2017kernel}, in which the spatial derivatives of a function were represented as a special kernel-based formulation. 
Here, we designed a new kernel-based approach of the spatial derivatives, which can maintain the good properties of the original one, such as the high order accuracy and unconditionally stable for one-dimensional problems when coupling with the high order explicit strong-stability-preserving Runge-Kutta method in time. Hence, it allowed much larger time step evolution compared with other explicit schemes.
In additional, without extra computational cost comparing with the old methods, the available interval of the special parameter $\beta$ in the formula is much larger, resulting in less errors and higher efficiency. 
Moreover, theoretical investigations indicated that the proposed scheme is unconditionally stable for multi-dimensional problems, which cannot be established with the previous methods.
A collection of numerical tests verified the performance of the proposed
scheme, demonstrating both its designed high order accuracy and efficiency.
In the future, we plan to extend our schemes to solve the time-dependent problems with general boundary conditions. And other time discretization would be considered as well. 

\appendix
\section{Coefficients in quadrature} \label{appen1}

Here, we list the coefficients in quadrature with fifth order accuracy \eqref{eq:quar}. Denote $\nu=\alpha\Delta x$. Then, we have
\begin{align*}
	C_{j-3}&=-\frac{60-15\nu^2+2\nu^4-(60+60\nu+15\nu^2-5\nu^3-3\nu^4)e^{-\nu}}{60\nu^5}, \\
	C_{j-2}&=\frac{120+24\nu-42\nu^2-2\nu^3+6\nu^4-(120+144\nu+42\nu^2-12\nu^3-8\nu^4)e^{-\nu}}{24\nu^5}, \\
	C_{j-1}&=-\frac{120+48\nu-42\nu^2-16\nu^3+12\nu^4-(120+168\nu+66\nu^2-14\nu^3-12\nu^4) e^{-\nu}}{12\nu^5}, \\
	C_{j}&=\frac{120\nu+132\nu^2+26\nu^3}{12\nu^5}, \\
	C_{j+1}&=-\frac{120+96\nu+6\nu^2-32\nu^3-12\nu^4-(120+216\nu+150\nu^2+30\nu^3-26\nu^4+24\nu^5)e^{-\nu}}{24\nu^5}, \\
	C_{j+2}&=\frac{60+60\nu+15\nu^2-5\nu^3-3\nu^4-(60+120\nu+105\nu^2+50\nu^3+12\nu^4)e^{-\nu}}{60\nu^5}.
\end{align*}

\bibliographystyle{abbrv}
\bibliography{ref}

\end{document}